\documentclass[12pt]{amsart}
\usepackage[matrix,arrow]{xy}
\usepackage{amssymb}
\usepackage{amsmath}
\usepackage{amsbsy}
\usepackage{bm}
\usepackage{amsfonts}
\usepackage{epsfig}
\usepackage{color}
\usepackage{epstopdf}
\usepackage{graphicx}
\usepackage{amsthm}
\usepackage{enumerate}
\usepackage[mathscr]{eucal}
\usepackage{verbatim}
\usepackage{braket}
\usepackage{mathtools}

\usepackage[bookmarks=true,hyperindex,pdftex,colorlinks,citecolor=blue,
linkcolor=blue,urlcolor=blue]{hyperref}
\usepackage{tikz}
\usetikzlibrary{matrix,arrows.meta}

\theoremstyle{plain}
\newtheorem{theorem}{Theorem}[section]
\newtheorem{cor}[theorem]{Corollary}
\newtheorem{prop}[theorem]{Proposition}

\newtheorem{lemma}[theorem]{Lemma}

\theoremstyle{definition}

\newtheorem{rem}[theorem]{Remark}
\newtheorem{remark}[theorem]{Remark}
\newtheorem{definition}[theorem]{Definition}

\newcommand\myeq{\stackrel{\mathclap{\tiny \normalfont\mbox{1}}}{=}}

\newcommand{\R}{\mathbb{R}}
\newcommand{\N}{\mathbb{N}}

\newcommand{\eps}{\varepsilon}

\DeclareMathOperator{\supp}{supp}

\DeclareMathOperator{\sgn}{sgn}

\DeclareMathOperator{\spann}{span}
\DeclareMathOperator{\diam}{diam}
\DeclareMathOperator{\co}{co}

\DeclareMathOperator{\Id}{Id}
\DeclareMathOperator{\dc}{dc}
\DeclareMathOperator{\wdc}{\textit{w}*dc}
\DeclareMathOperator{\wdec}{\textit{w}*{\delta}c}
\DeclareMathOperator{\sdc}{sdc}
\DeclareMathOperator{\dec}{{\delta}c}
\DeclareMathOperator{\dent}{dent}

\DeclareMathOperator{\ext}{ext}

\DeclareFontFamily{U}{mathb}{}
\DeclareFontShape{U}{mathb}{m}{n}{<-5.5> mathb5 <5.5-6.5> mathb6 
<6.5-7.5> mathb7 <7.5-8.5> mathb8 <8.5-9.5> mathb9 <9.5-11> mathb10 
<11-> mathb12}{}
\DeclareSymbolFont{mathb}{U}{mathb}{m}{n}
\DeclareFontSubstitution{U}{mathb}{m}{n}

\DeclareMathSymbol{\blackdiamond}{\mathbin}{mathb}{"0C}  % \mathbin for a binary operator

\newcommand{\Lip}{{\mathrm{Lip}}_0}

\title[The Daugavet and Delta-constants of points in Banach spaces]
{The Daugavet and Delta-constants of points in Banach spaces}

\author[G.~Choi]{Geunsu Choi}
\address[G.~Choi]{Department of Mathematics Education, Sunchon National University, 57922 Jeollanam-do, Republic of Korea}
\email{\texttt{gschoi@scnu.ac.kr}}

\author[M.~Jung]{Mingu Jung} 
\address[M.~Jung]{June E Huh Center for Mathematical Challenges, Korea Institute for Advanced Study, 02455 Seoul, Republic of Korea}
\email{\texttt{jmingoo@kias.re.kr}}


\date{\today}                                           

\begin{document}

\begin{abstract}
We introduce two new notions called the \textit{Daugavet constant} and \textit{$\Delta$-constant} of a point, which measure quantitatively how far the point is from being Daugavet point and $\Delta$-point and allow us to study Daugavet and $\Delta$-points in Banach spaces from a quantitative viewpoint. We show that these notions can be viewed as a localized version of certain global estimations of Daugavet and diametral local diameter two properties such as Daugavet indices of thickness. As an intriguing example, we present the existence of a Banach space $X$ in which all points on the unit sphere have positive Daugavet constants despite the Daugavet indices of thickness of $X$ being zero. 
Moreover, using the Daugavet and $\Delta$-constants of points in the unit sphere, we describe the existence of almost Daugavet and $\Delta$-points as well as the set of denting points of the unit ball. 
We also present exact values of the Daugavet and $\Delta$-constant on several classical Banach spaces, as well as Lipschitz-free spaces. In particular, it is shown that there is a Lipschitz-free space with a $\Delta$-point which is the furthest away from being a Daugavet point. Finally, we provide some related stability results concerning the Daugavet and $\Delta$-constant.
\end{abstract}

\maketitle
	
	\hypersetup{linkcolor=black}

\makeatletter \def\l@subsection{\@tocline{2}{0pt}{1pc}{5pc}{}} \def\l@subsection{\@tocline{2}{0pt}{3pc}{6pc}{}} \makeatother

\tableofcontents

\hypersetup{linkcolor=blue}

\section{Introduction}

Let $X$ be a Banach space with closed unit ball $B_X$, unit sphere $S_X$, and topological dual space $X^*$. 
We say that $X$ has the \textit{Daugavet property} if the equation
\begin{equation}
\| \Id + T \| = 1+\|T\| \tag{DE} \label{eq:DE}
\end{equation} 
holds for every rank-one bounded linear operator from $X$ to itself, where $\Id$ denotes the identity operator on $X$. Given $x \in S_X$ and $\eps >0$, let us denote by $\Delta_\eps (x) = \{ y \in B_X : \| y- x\| \geq 2 -\eps\}$. By the geometric characterization in \cite{KSSW}, $X$ has the Daugavet property if and only if for every $x \in S_X$ and $\eps >0$ we have $B_X \subseteq \overline{\co} \,\Delta_{\eps}(x)$.

The following related (weaker) property is introduced in \cite{IK}: a Banach space $X$ is \textit{with bad projections} if for every rank-one projection $P$, we have
\begin{equation}
\| \Id - P \| \geq 2. \tag{$\Delta$E} \label{eq:BP}
\end{equation}
Let us mention that this property was studied in \cite{AHNTT, BLR18} under the names of \textit{local diameter two property $+$} and \textit{diametral local diameter two property}. There is a similar characterization for this property \cite{IK}: a Banach space $X$ is with bad projections if and only if  for every $x \in S_X$ and $\eps >0$, we have $x \in \overline{\co} \,\Delta_{\eps}(x)$. 

There have been numerous efforts in the study of the Daugavet property and its related properties as it turned out that they could significantly affect the isomorphic structure of a Banach space. For instance, a Banach space with the Daugavet property contains an isomorphic copy of $\ell_1$ \cite{KSSW}, it does not have unconditional basis \cite{K96}. Moreover, if a Banach space is with bad projections and has an unconditional basis, then the unconditional suppression basis constant must be at least 2. 
Classical examples of Banach spaces having the Daugavet property include $C(K), L_1 (\mu)$ and $L_\infty (\mu)$ provided that $K$ is perfect and $\mu$ is non-atomic \cite{D63,Loz, W96}, the space of Lipschitz functions $\Lip (M)$ over a metrically convex space $M$ \cite{ikw} and the Banach algebra of bounded holomorphic functions $H^\infty$ \cite{J23, W97, Wo92}. 
For background on the Daugavet property and its related properties including being with bad projections, we refer to \cite{KW04, MR2022, W01} and references therein. 

The pointwise version of the Daugavet property and the diametral local diameter two property were first introduced in \cite{AHLP} and studied in \cite{ALMT, DJRZ, JRZ, V2023, V_lipschitzfunction}: given a Banach space $X$ and a point $x \in S_X$, $x$ is said to be 
\begin{itemize}
\itemsep0.3em
\item a \textit{Daugavet point} if for every $\eps>0$, $\delta>0$ and $x^* \in S_{X^*}$, there exists $y \in S(x^*,\delta)$ such that $\|y-x\| > 2 - \eps$.
\item a \textit{$\Delta$-point} if for every $\eps>0$, $\delta>0$ and $x^* \in S_{X^*}$ with $x \in S(x^*,\delta)$, there exists $y \in S(x^*,\delta)$ such that $\|y-x\| > 2 - \eps$,
\end{itemize} 
where $S(x^*,\delta)$ is a \textit{slice} of $B_X$, i.e., $S(x^*,\delta)=\{y \in B_X : \text{Re}\,x^* (y) > 1-\delta\}$. For the spaces $L_1 (\mu)$ and $L_1$-predual spaces, the notions of Daugavet point and $\Delta$-point coincide \cite{AHLP, MR2022}. However, in general, there exists a $\Delta$-point which is not a Daugavet point (see, for instance, \cite[Example 4.7]{AHLP} and \cite[Example 4.4]{JRZ}).

In this article, we introduce the concepts of Daugavet constant and $\Delta$-constant (see Definition \ref{def:dc_and_dec} for definitions) of points in order to analyze how close or how far the points in Banach spaces are from becoming Daugavet points or $\Delta$-points.
The aim of this work is to initiate a systematic investigation of Daugavet and $\Delta$-constants as well as to extend several known properties and results on Daugavet and $\Delta$-points from a quantitative point of view.

The Daugavet index of thickness, which quantifies how far a Banach space is from having the Daugavet property, was first introduced in \cite{RZ2018} and its variations are deeply studied in \cite{HLLNR}.
%, and this allowed them to solve an open question whether every slice always have diameter two or not when they are all of radius one. 
Motivated by the question of whether or not a super-reflexive space admits $\Delta$-points, the authors of \cite{ALMP} investigated a Banach space that admits \textit{almost Daugavet (or, $\Delta$-) points}. 
These are weaker than assuming the existence of Daugavet and $\Delta$-points in Banach spaces. 
For instance, $c_0$ does not have any $\Delta$-points, while it admits almost $\Delta$-points. It turns out that Daugavet indices of thickness and the existence of almost Dauagvet or $\Delta$-points are closely related to the quantitative approach of Daugavet and $\Delta$-points, so these will be one of our main aim of the study.

Throughout the paper, $\ext(B_X)$ and $\dent(B_X)$ denote the set of all extreme points and all denting points, respectively, of $B_X$, where a point $x \in S_X$ is said to be a \emph{denting point} if there is a slice of $B_X$ containing $x$ with an arbitrary small diameter. Recall that a Banach space $X$ has the \emph{Radon-Nikod\'ym property} if every nonempty bounded subset $B \subseteq X$ is \textit{dentable}, that is, $B$ contains a slice with an arbitrary small diameter. It is immediate that a Banach space with Daugavet property (or, diametral local diameter two property) fails the Radon-Nikod\'ym property as it has no denting points. In contrast to this, a surprising example of a Banach space with the Radon-Nikod\'ym property and a Daugavet point was constructed in \cite{V2023}. Recently, it turned out that every infinite-dimensional Banach space can be renormed with a $\Delta$-point \cite{AAL+}, and a Daugavet point provided that it has an unconditional weakly null Schauder basis \cite{HLPV} (hence, $\ell_2$ can be renormed with a Daugavet point). Let us also mention that a Banach space with a subsymmetric basis cannot have $\Delta$-points \cite{ALMT}.

The outline of the paper is as follows. In Section \ref{sec:general}, we investigate the general properties of Daugavet and $\Delta$-constants. We give the definitions of Daugavet and $\Delta$-constant and obtain some general observations on them in Section \ref{subsec:def}. Moreover, we study their relation with some operator-norm inequalities which extends the well-known relation \cite{AHLP} between Daugavet points (resp., $\Delta$-points) and the equality \eqref{eq:DE} (resp., \eqref{eq:BP}). 
Section \ref{subsec:convex} is devoted to study the relation between our new notions and extremal points of the unit ball. Namely, we show that the set of points with vanishing $\Delta$-constant coincides with the set of denting points of $B_X$. 
Furthermore, we provide precise values of the Daugavet constant of a locally uniformly rotund point of the unit ball and $\Delta$-constant of a point lying in the unit ball of a Hilbert space. 
In Section \ref{subsec:global}, we link our two new notions to global estimations of Daugavet and diametral local diameter two properties introduced in \cite{HLLNR, RZ2018}. 
More precisely, we first observe that the characterization for $X$ to admit almost Daugavet points can be done using the supremum of Daugavet constants of $x$ over all points in $S_X$. Also, the infimum of Daugavet constants of $x$ over all points in $S_X$ is shown to be, by its nature, the same as the Daugavet index of thickness of $X$. Finally, we construct a Banach space whose infimum of $\Delta$-constants is zero while there are no points with vanishing Daugavet constants.

In Section \ref{sec:classical}, we present a number of estimations on Daugavet and $\Delta$-constants of a point in certain classical Banach space including $c_0$, $L_1$-spaces, uniform algebras and Lipschitz-free spaces. We present some optimal calculations of those values, while for others, we discuss possible bounds on the values that the Daugavet and $\Delta$-constant can have. It is observed as a consequence that behaviors of Daugavet constant and $\Delta$-constant are totally different in the case of $c_0$ spaces, while those two constants coincide in  certain $L_1$-spaces. Finally, through the ideas used to study Daugavet and $\Delta$-points in Lipschitz-free spaces, we are able to construct a Banach space with a $\Delta$-point whose Daugavet constant is zero.

In Section \ref{sec:stability}, we explore some related stability results. These results from a quantitative standpoint generalize several stability results that have already been discussed in many literatures. As an application, we estimate lower bounds for Daugavet and $\Delta$-constants of points in spaces of absolute sums such as $\ell_1$- and $\ell_\infty$-sums.

\section{General results on Daugavet and $\Delta$-constants}\label{sec:general}

\subsection{The Daugavet and $\Delta$-constant}\label{subsec:def}

In the following all Banach spaces are assumed to be over the \textbf{real} field. 

\begin{definition}\label{def:dc_and_dec}
Let $X$ be a Banach space and $x \in B_X$. 
\begin{itemize}
\itemsep0.3em
\item 
The \emph{Daugavet constant} of the point $x$ is 
$$
\dc(x) := \sup \Set{ \alpha \geq 0 | \begin{array}{l} \forall \eps,\delta>0, y^* \in S_{X^*},\,\exists \, y \in S(y^*,\delta) \\
\text{ such that } \|y-x\| > \alpha - \eps \end{array}}.
$$
\item 
The \emph{$\Delta$-constant} of the point $x$ is 
$$
\dec(x) := \sup \Set{ \alpha\geq 0 | \begin{array}{l} \forall \eps,\delta>0, y^* \in S_{X^*} \text{ with } x \in S(y^*,\delta), \\ 
\exists \, y \in S(y^*,\delta) \text{ so that } \|y-x\| > \alpha - \eps \end{array}}.
$$
\end{itemize} 
\end{definition}
Notice that the definition remains equivalent without the use of $\varepsilon$. However, we include it in the definition so that the supremums are indeed maximums. Also, it is clear by definition that $x \in S_X$ is a Daugavet point (resp. $\Delta$-point) if and only if $\dc(x) = 2$ (resp. $\dec(x) =2$). %The notions of Daugavet and $\Delta$-constants allow us to estimate how far or close a given point is from becoming Daugavet and $\Delta$-point. 

Observe from the Definition \ref{def:dc_and_dec} of Daugavet and $\Delta$-constant that  
\[
1-\|x\| \leq \dc(x) \leq \dec(x) \leq 1+\|x\|
 \]
and $\dc(x) = \dc (-x)$ and $\dec (x) = \dec(-x)$ for every $x \in B_X$. 
By definition, we have that 
\begin{equation}\label{eq:dc_ball}
\sup_{y\in S} \|x-y\| \geq \dc(x) \,  \text{ for any slice } S \text{ of } B_X. 
\end{equation} 
On the other hand, for any slice $S_0$ of $B_X$ and $\eps >0$, there exists $y_0 \in S_0$ so that $\|x-y_0\| \geq \sup_{y\in S_0} \|x-y\| - \eps.$
This, combined with \eqref{eq:dc_ball}, shows that 
\begin{equation*}\label{eq:dc_ball2}
\dc(x) = \inf_{\substack{S \subseteq B_X \text{ slice }} } \sup_{y \in S} \|x-y\|. 
\end{equation*}
Similarly, one can show that 
\begin{equation*}\label{eq:dc_ball3}
\dec(x) = \inf_{\substack{S \subseteq B_X \text{ slice } \\ x \in S} } \sup_{y \in S} \|x-y\|. 
\end{equation*}
In this regard, following the terminology in \cite[Section 5]{MPR2023}, we can say a point $x \in B_X$ is a Daugavet point (resp. $\Delta$-point) if and only if $\dc(x) = 1+\|x\|$ (resp. $\dec(x) = 1+\|x\|$).

\begin{rem}\label{rem:dc_Lipschitz} 
Note that $x \in B_X \mapsto \dc(x)$ is a $1$-Lipschitz map. In fact, let $x, y \in B_X$ be given. Given a slice $S$ and $\eps$, choose $u \in S$ so that $\| x - u \| > \dc(x) - \eps$. Then
\[
\| y -u\| \geq \|x - u \| - \| x - y\| > \dc (x) - \| x-y\| -\eps.  
\]
This implies that $\dc(y) \geq \dc(x) - \|x-y\|$. By changing the role of $x$ and $y$, we obtain $\dc(x)\geq \dc(y) - \|x-y\|$. It follows that $|\dc(x)-\dc(y)| \leq \|x-y\|$.
\end{rem} 

%Notice that a point $x \in S_X$ is a Daugavet point (resp. $\Delta$-point) if and only if $\dc(x)=2$ (resp. $\dec(x)=2$). As a consequence,  

%\subsection{On basic properties and general results}\label{subsec:general}

Let us observe the following as a direct application of the Hahn-Banach theorem, which may be considered as a quantitative counterpart of the geometric characterization of being a Daugavet and a $\Delta$-point.

\begin{prop}\label{prop:dc-characterization}
Let $X$ be a Banach space and $0 < \alpha \leq 2$, $x \in B_X$ be given.
\begin{enumerate}
\itemsep0.3em
\item[\textup{(a)}] $\dc(x) \geq \alpha$ if and only if $B_X \subseteq \overline{\co} \,\Delta_{2-\alpha+\eps}(x)$ for every $\eps >0$. 
\item[\textup{(b)}] $\dec(x) \geq \alpha$ if and only if $x \in \overline{\co} \,\Delta_{2-\alpha+\eps}(x)$ for every $\eps >0$. 
\end{enumerate}
\end{prop}

Recall from \cite{AHLP} that: 
\begin{itemize}
\itemsep0.3em
\item 
$x \in S_X$ is a Daugavet point if and only if for every non-zero $x^* \in X^*$, the rank-one operator $T = x^* \otimes x$ satisfies $\|\Id + T\| = 1+\|T\|$. 
\item
$x \in S_X$ is a $\Delta$-point if and only if for every $x^* \in X^*$ with $x^* (x)=1$, the projection $P=x^* \otimes x$ satisfies $\| \Id - P\| \geq 2$.
\end{itemize}
In the following theorem, we present generalizations of these results in terms of the Daugavet constant and the $\Delta$-constant. 

\begin{prop}\label{prop:dc-operator}
Let $X$ be a Banach space with $\dim (X) \geq 2$ and $x \in B_X$.  
\begin{enumerate}
\itemsep0.3em
\item[\textup{(a)}] For every rank one operator $T = x^* \otimes x$ with $\| x^* \| \geq 1$, we have 
\[
\| \Id + T \| \geq  (\dc(x)  - 1)\|x^*\| + 1. 
\] 
\item[\textup{(b)}] For every rank one projection $P = x^* \otimes x$ with $x^* \in {X^*}$ satisfying $x^* (x)=1$, we have 
\[
\| \Id - P \| \geq \dec (x). 
\] 
\end{enumerate}
\end{prop} 

\begin{proof}
(a): If $\dc(x) \leq 1$, then the inequality is trivial since
\[
\|\Id + T\| \geq 1 \geq (\dc(x) - 1)\|x^*\| + 1,
\]
so we may assume that $\dc(x) > 1$. Let $x^* \in {X^*}$ with $\|x^*\| \geq 1$, $\delta >0$ and $0 < \eps < \dc(x) -1$ be given. Suppose first that $\|x^*\| = 1$. Since $\dc(-x)=\dc(x)$, we can find $y \in S(x^*, \delta )$ so that $\|y+x\| > \dc(x) - \eps> 1$. This implies that 
\begin{align*}
\| \Id + x^* \otimes x\| &\geq \|y + x^* (y) x\| \geq \|y+x\| - \delta > \dc(x) - \eps - \delta. 
\end{align*}
Since $\delta >0$ and $\eps >0$ were chosen arbitrarily, we conclude that $\| \Id + x^* \otimes x\| \geq \dc(x)$.

Next, suppose that $\|x^*\| >1$. Fix $y \in S\bigl(\frac{x^*}{\|x^*\|}, \delta \bigr)$ so that $\|x+y\| > \dc(x) - \eps> 1$. For a small enough $\delta>0$ such that $x^*(y) > 1$, we have
\begin{align*}
\| \Id + x^* \otimes x\| &\geq \|y + x^* (y) x\| \geq x^*(y) \|y+x\| -x^*(y) +1 \\
&> x^*(y) (\dc(x) - \eps -1) +1 > \|x^*\| (1-\delta) (\dc(x) - \eps -1) +1. 
\end{align*}
Again from the choice of $\delta >0$ and $\eps >0$, it follows that $\| \Id + x^* \otimes x\| \geq (\dc(x)-1) \|x^*\| +1$.

(b): First note that in at least $2$-dimensional Banach space $X$, the inequality is trivial if $\dec(x) \leq 1$. For $\dec(x) > 1$ fix $\varepsilon, \delta>0$ small and find $y \in B_X$ such that $x^*(y) > 1- \delta$ and $\|x- y\| > \dec(x) -\varepsilon>1$. If $x^*(y) \leq 1$, then
\[
\|\text{Id} - P \| \geq \|y - x^*(y) x\| \geq \|y - x\| - \delta > \dec(x) - \varepsilon - \delta.
\]
If $x^*(y)>1$, then
\begin{align*}
\|\text{Id} - P \| &\geq \|y - x^*(y) x\| \\ 
&\geq x^*(y) \|x-y\| - x^*(y) +1 \\
&> x^*(y) ( \dec(x) - \varepsilon - 1) + 1 > (1-\delta)(\dec(x) - \varepsilon - 1) + 1.
\end{align*}
As $\varepsilon$ and $\delta$ were arbitrary, this concludes the proof.
\end{proof} 

In order to show that the notion of Daugavet and $\Delta$-points coincide for an $L_1$-predual space, the equality $S_X \cap \big( \cap_{\eps>0}\, \overline{\co} \, \Delta_\eps^X (x) \big) = S_X \cap \big( \cap_{\eps>0}\, \overline{\co} \, \Delta_\eps^{X^{**}} (x) \big)$ for any $x \in S_X$ is proved in \cite[Lemma 3.5]{AHLP}. 
Here, $\Delta_\eps^X (x)$ is the set $\Delta_\eps (x)$ and $\Delta_\eps^{X^{**}} (x) = \{ z \in B_{X^{**}} : \|z-x\| \geq 2- \eps\}$.  
We would like to extend this result by replacing the constant $2$ in the definition of $\Delta_\eps$ with $\alpha \in (0,2]$.  

Given $x \in B_X$, the notation $\dc_{X^{**}} (x)$ and $\dec_{X^{**}} (x)$ denotes the Daugavet constant and $\Delta$-constant, respectively, of the point $x$ where $x$ is considered as an element of $X^{**}$. Just to emphasize, let us denote by $\dc_X (x) = \dc (x)$ and $\dec_X (x) = \dec (x)$ in the proposition below.

\begin{prop}\label{prop:PLR}
Let $X$ be a Banach space, $x \in B_X$ and $0< \alpha \leq 2$. Then 
\begin{equation}\label{eq:bidual}
B_X \cap \overline{\co} \,\Delta_{2-\alpha}^{X^{**}} (x) \subseteq \cap_{\eps>0}\, \overline{\co} \,\Delta_{2-\alpha +\eps}^{X} (x). 
\end{equation} 
In particular, we have the following: 
 \begin{enumerate}
 \itemsep0.3em
\item[\textup{(a)}] $\dc (x) \geq \dc_{X^{**}} (x)$.  
\item[\textup{(b)}] $\dec (x) =\dec_{X^{**}} (x)$.
% If $\dc_{X^{**}} (x) \geq \alpha$, then $\dc_X (x) \geq \alpha$. 
\end{enumerate}
\end{prop} 

%\begin{proof}
%
%\end{proof} 

\begin{proof}
The inclusion \eqref{eq:bidual} can be proved similarly as in \cite[Lemma 3.5]{AHLP}. For the sake of completeness, we present the proof. Let $\eps >0$ be given. Let $y \in B_X$, $y_1^{**}, \ldots, y_{m}^{**} \in \Delta_{2-\alpha}^{X^{**}} (x)$ and $\lambda_1,\ldots, \lambda_m > 0$ be such that $\sum_{i=1}^m \lambda_i = 1$ and $\|y - \sum_{i=1}^m \lambda_i y_i^{**} \| < \frac{\eps}{2}$. Put $E = \spann \{ x, y, y_1^{**}, \ldots, y_m^{**} \} \subseteq X^{**}$ and let $0<\eta<1$ be such that $(1-\eta)\alpha > \alpha-\eps$. By Principle of Local Reflexivity, there exists a linear operator $T : E \rightarrow X$ such that 
 \begin{itemize}
 \itemsep0.3em
\item $Te = e$ for every $e \in E \cap X$; 
\item $(1-\eta)\|e\| \leq \|Te\| \leq (1+\eta) \| e\|$ for every $e \in E$. 
\end{itemize} 
It follows that $\|x - Ty_{i}^{**}\| = \|Tx - Ty_i^{**}\| \geq (1-\eta) \alpha > \alpha - \eps$. That is, $Ty_i^{**} \in \Delta_{2-\alpha+\eps}^X (x)$ for $i=1,\ldots, m$. 
Moreover, 
\begin{align*}
\left\| y - \sum_{i=1}^m \lambda_i T y_i^{**} \right\| = \left\| Ty - \sum_{i=1}^m \lambda_i T y_i^{**}\right\| \leq (1+\eta)\frac{\eps}{2} < \eps.  
\end{align*} 
It follows that $y \in \overline{\co} \,\Delta_{2-\alpha +\eps}^{X} (x)$.

(a): By Proposition \ref{prop:dc-characterization}, if $\dc_{X^{**}} (x) \geq \alpha$, then $B_{X^{**}} \subseteq \overline{\co}  \Delta_{2-\alpha+\eps}^{X^{**}} (x)$ for every $\eps >0$. Let us fix $\eps >0$. Note that if $y \in B_X$, then $y \in \overline{\co}  \Delta_{2-\alpha+\eps}^{X^{**}} (x)$. By \eqref{eq:bidual}, we have that $y \in \overline{\co} \Delta_{2-\alpha+\eps+ \eta}^{X} (x)$ for every $\eta >0$. In particular, $y \in \overline{\co} \Delta_{2-\alpha+2\eps}^{X} (x)$ which shows that $B_X \subseteq \Delta_{2-\alpha+2\eps}^{X} (x)$. As this holds for every $\eps >0$, by Proposition \ref{prop:dc-characterization}, we conclude that $\dc (x) \geq \alpha$.  

(b): Note from Proposition \ref{prop:dc-characterization} that $x \in \overline{\co} \Delta_{2-\dec(x)+\eps}^{X} (x)$ for every $\eps >0$, which is, by \eqref{eq:bidual}, equivalent to that $x \in \overline{\co} \Delta_{2-\dec(x)+\eps}^{X^{**}} (x)$ for every $\eps >0$. It follows that $\dec (x) \geq \alpha$ if and only if $\dec_{X^{**}}(x) \geq \alpha$. 
\end{proof} 

\begin{rem}
Let $Y$ be a subspace of $X$ and let $x \in B_Y$. Let us denote by $\dec_Y (x)$ the $\Delta$-constant of $x$ when the point $x$ is considered as a point of $Y$. Given $x \in B_Y$, $\eps >0$ and a slice $S(x^*,\delta)$ of $B_X$ containing $x$, consider $x^* \vert_Y \in Y^*$ which satisfies clearly that $x^*\vert_Y (x) > 1-\delta$. Then we may find $u \in B_Y \cap S(x^*, \delta)$ such that $\|u-x\| > \dec_Y (x) -\eps$. This implies that $\dec_X (x) \geq \dec_Y (x)$ which, in particular, shows that $\dec_{X^{**}} (x) \geq \dec_X (x)$ and that $\Delta$-points pass to superspaces. 
\end{rem}

\subsection{Relationship to extremal points}\label{subsec:convex}

By nature, the concept of a $\Delta$-point can be understood as being diametrically opposed to that of a denting point. The following confirms the intuition by showing that the set of denting points is precisely the set of points with vanishing $\Delta$-constant.

%\textcolor{blue}{We have a bit more interesting situation for Banach spaces with the \textup{RNP}. In fact, it is shown that $\sup_{x \in S_{\ell_1}} \dc(x)=2$ \cite[Lemma 6.6]{ALMP} while $\dec(x)<2$ for every $x \in S_{\ell_1}$, and there even exists a Banach space $X$ with the \textup{RNP} such that $\dc(x)=2$ for some $x \in S_X$ \cite[Example 3.1]{V2023}. }

\begin{prop}\label{prop:denting_dec}
Let $X$ be a Banach space. Then, we have
$$
\dent(B_X) = \{x \in S_X : \dec(x) = 0 \}.
$$
\end{prop}

\begin{proof}
Suppose first that $x \in S_X$ satisfies $\dec(x)=0$. Let $\eta>0$ be given. Then, there exist $\eps>0$, $\delta>0$ and $y^* \in S_{X^*}$ such that $x \in S(y^*,\delta)$ and $\|y-x\| \leq \eta - \eps$ for every $y \in S(y^*,\delta)$. Thus $S(y^*,\delta)$ is a slice of $B_X$ containing $x$ such that $\diam S(y^*,\delta) \leq 2(\eta-\eps)$. Since $\eta>0$ was arbitrary, it follows that $x$ is a denting point of $B_X$.

On the other hand, let $x \in \dent(B_X)$ and $\eta>0$ be given. Fix $\delta>0$ and $y^* \in S_{X^*}$ such that $x \in S(y^*,\delta)$ and $\diam S(y^*,\delta) < \eta/2$. Then $\|y-x\| \leq \eta/2$ for every $y \in S(y^*,\delta)$, so $\dec(x) \leq \eta$. This finishes the proof.
\end{proof}

We refer to Theorem \ref{theorem:deltanotdauga} for the remark that the same kind of result cannot be guaranteed in the case of Daugavet constants. On the other hand, it is observed in \cite[Proposition 3.1]{JRZ} that Daugavet points have to be far from the set of denting points. In general, this result can be interpreted as follows by using the Daugavet constant, which will be used in Section \ref{subsec:lipfree}.

\begin{lemma}\label{lemma:dentdc}
Let $X$ be a Banach space and $x \in B_X$. If $y \in \dent(B_X)$, then $\|x-y\| \geq \dc(x)$. 
\end{lemma}

\begin{proof}
Suppose that $y \in \dent(B_X)$ and $\|x-y\| < \dc(x)$. Take a slice $S = S(x^*, \delta)$ containing $y$ such that $\diam (S) < \eta$, where $0< 2\eta < \dc(x) - \|x-y\|$. Find $u \in S$ so that $\|x-u\| > \dc(x) -\eta$. Then 
\[
\|x-y\| \geq \|x-u \| - \|u-y\| > \dc(x) - \eta -\eta > \|x-y\|,
\]
which is a contradiction.
\end{proof}

%NOTE: What can we say about the set $\{ x \in S_X : \dc (x) =0\}$? Is there a point $x \in S_X$ satisfying that $\dc(x) = 0$ but $\dec(x) >0$? If a point $x$ is an extreme point of $B_X$, then $\dc (x) =0$? See Section XXX. 

Recall that a Banach space $X$ is said to have the \textit{Kadec property} if the norm topology and the weak topology coincide on $S_X$. It is known that $X$ is strictly convex with the Kadec property if and only if every point in $S_X$ is a denting point of $B_X$ \cite{LLT86}. It is well known that $X$ has the Radon-Nikod\'ym property if and only if every bounded closed convex subset of $X$ is the closed convex hull of its denting points. As a direct consequence of Proposition \ref{prop:denting_dec}, we obtain the following.

\begin{cor} Let $X$ be a Banach space.
\begin{enumerate}
\itemsep0.3em
\item [\textup{(a)}] If $X$ is strictly convex with the Kadec property \textup{(}in particular, $X$ is locally uniformly rotund\textup{)}, then $\dec(x)=0$ for every $x \in S_X$.
\item [\textup{(b)}] If $B_X$ is dentable, then, we have $\inf_{x \in S_X} \dec(x)=0$. Moreover, if $X$ has the Radon-Nikod\'ym property, then $B_X = \overline{\co}\{x \in S_X : \dec(x) = 0 \}$.
\end{enumerate} 
\end{cor}

In the case of locally uniformly rotund (for short, LUR) points $x \in B_X$, we can compute the exact value of $\dc(x)$. 
Recall that $x \in B_X$ is a \emph{\textup{LUR} point} if $\|x_n - x \| \rightarrow 0$ whenever a sequence $(x_n) \subseteq X$ with $\|x_n\| \rightarrow \|x\|$ satisfies that $ \|x_n+x\| \rightarrow 2\|x\|$. Observe that $x \in X \setminus \{0\}$ is a LUR point if and only if $\| x_n - x/\|x\| \| \rightarrow 0$ whenever a sequence $(x_n) \subseteq S_X$ satisfies that $ \|x_n + x\| \rightarrow 1 + \|x\|$.

\begin{prop}
Let $X$ be a Banach space, and let $x \in B_X$ be a \textup{LUR} point. Then, we have $\dc(x) = 1-\|x\|$.
\end{prop}

\begin{proof}
It is clear that $\dc(x) \geq 1- \|x\|$ from the definition. In order to show that $\dc(x) \leq 1 - \|x\|$, let $\eps>0$ be given and fix $x^* \in S_{X^*}$ such that $x^*(x) = \|x\|$. We claim that there exists $\delta>0$ such that if $x^*(y) > 1-\delta$ for some $y \in B_X$, then $\bigl\|y-x/\|x\|\bigr\| < \eps$. Otherwise, we may choose $y_n \in S(x^*, 1/n)$ for each $n \in \N$ such that $\bigl\| y_n - x/\|x\| \bigr\| \geq \eps$ for all $n \in \N$. But notice that
$$
\|y_n+x\| \geq x^*(y_n) +x^*(x) > 1+\|x\|-\frac{1}{n};
$$
thus, since $x$ is a LUR point, we can deduce that $\|y_n - x/\|x\|\| \rightarrow 0$. This contradicts the choice of $(y_n)$. Now, observe that for every $y \in S(x^*,\delta)$,
$$
\|y-x\| \leq \left\|y - \frac{x}{\|x\|}\right\| + 1- \|x\| < 1-\|x\|+\eps.
$$
Since $\eps>0$ was arbitrary, this shows that $\dc(x) \leq 1 - \|x\|$.
\end{proof}

Dealing with Hilbert spaces, we can compute the $\Delta$-constant of points in the unit ball. Note that for the $1$-dimensional Hilbert space $H$, it is obvious that $\dec(x) = 1-\|x\|$ for $x \in B_H$. 

\begin{prop}
Let $H$ be a Hilbert space with $\dim(H) \geq 2$, and $x \in B_H$. Then, we have $\dec(x) = \sqrt{1-\|x\|^2}$.
\end{prop}

\begin{proof}
Consider the slice $S=\{ u \in B_H : \big\langle u, \frac{x}{\|x\|} \big\rangle > \|x\| - \delta\}$ with $\delta >0$, where $\langle \cdot, \cdot \rangle$ denotes the inner product on $H$. It is clear that $x \in S$. For $y \in S$, observe that 
\begin{equation}\label{eq:inner}
\|x-y\|^2 = \|x\|^2+\|y\|^2 - 2\langle x , y \rangle \leq \|x\|^2 + 1 - 2 \|x\| (\|x\|-\delta). 
\end{equation}
By letting $\delta \rightarrow 0$, the right-hand side of \eqref{eq:inner} tends to $1-\|x\|^2$. This implies that $\dec(x) \leq \sqrt{1-\|x\|^2}$.
For the reverse inequality, let $S(y, \alpha)$ be a slice containing the point $x$. Since $\dim(H) \geq 2$, we can find $z \in S_H$ which is orthogonal to $x$, i.e., $\langle z,x\rangle = 0$. By changing the sign if necessary, we may assume that $\langle z, y \rangle \geq 0$. Note that $x+  \sqrt{1-\|x\|^2 }\, z \in S_H$ and $x+  \sqrt{1-\|x\|^2 }\,z \in S(y, \alpha)$ since 
\[
\langle y, x+  \sqrt{1-\|x\|^2 } \, z \rangle \geq \langle y, x\rangle > 1 - \alpha.
\]
Note from $\| x -( x+  \sqrt{1-\|x\|^2 }\,z) \| = \sqrt{1-\|x\|^2 }$ that $\dec(x) \geq \sqrt{1-\|x\|^2 }$. 
\end{proof}

\subsection{On global estimations of Daugavet and $\Delta$-constants}\label{subsec:global}

Recall from \cite{ALMP} that $X$ \emph{admits almost Daugavet points \textup{(resp.} almost $\Delta$-points\textup{)}} if for every $\eps>0$, there exists $x \in B_X$ such that $B_X \subseteq \overline{\co} \,\Delta_\eps(x)$ (resp. $x \in \overline{\co} \,\Delta_\eps(x)$). From Proposition \ref{prop:dc-characterization}, we have the following characterization, which shows that the behavior of the Daugavet constant of points in the \textit{unit sphere} $S_X$ completely determines whether $X$ admits almost Daugavet points or not.%, and the same for the $\Delta$-constant. 

\begin{prop}\label{prop:almost}
Let $X$ be a Banach space.
Then $X$ admits almost Daugavet points if and only if $\sup_{x \in S_X} \dc(x)=2$.
%\item[\textup{(b)}] $X$ admits almost $\Delta$-points if and only if $\sup_{x \in S_X} \dec(x)=2$.
\end{prop}

%\begin{lemma}\label{lem:2-eps-Daugavet}
%Given $\eps >0$, if $x \in B_X$ is a $(2-\eps)$-Daugavet point, then $\dc(x) \geq 2-\eps$. 
%\end{lemma}

%\begin{prop}\label{prop:stability-infty}
%Let $X$ and $Y$ be Banach space and $Z = X \oplus_\infty Y$. Then, we have
%$$
%\sup_{z \in S_Z} \dc(z) \leq \max \left\{ \sup_{x \in S_X} \dc(x), \sup_{y \in S_Y} \dc(y) \right\}.
%$$
%\end{prop}

\begin{proof}%[Proof of Theorem \ref{theorem:almost}]
The proof of the ``if" part is clear. Assume that $X$ admits almost Daugavet points. Let $\eps >0$ and $x_0 \in B_X$ be so that $B_X \subseteq \overline{\co} \,\Delta_{\eps/2} (x_0)$. By Proposition \ref{prop:dc-characterization}, it is clear that $\dc (x_0) \geq 2-\frac{\eps}{2}$.
Note that $\|x_0\| \geq 1 - \frac{\eps}{2}$. Thus, putting $z:=\frac{x_0}{\|x_0\|} \in S_X$, we have that $\|z - x_0\| < \frac{\eps}{2}$. Since the mapping $x \mapsto \dc (x)$ is $1$-Lipschitz (see Remark \ref{rem:dc_Lipschitz}), we have
\[
\dc (z) \geq \dc (x_0) - | \dc(z) - \dc(x_0)| \geq \dc (x_0) - \|z-x_0\| > \left(2-\frac{\eps}{2}\right) - \frac{\eps}{2} = 2 - \eps; 
\]
which implies $\sup_{x \in S_X} \dc(x) =2$. 
\end{proof}

Let us point out that it is unclear whether or not an analog result for almost $\Delta$-points holds in general. 
In fact, given a point $x \in B_X$ satisfying that $x \in \overline{\co} \, \Delta _\eps (x)$, we cannot expect the normalization $z= \|x\|^{-1} x$ of the point $x$ to satisfy $z \in \overline{\co} \, \Delta _\eps (z)$. Moreover, the mapping $x \in B_X \mapsto \dec (x)$ is not a Lipschitz map in general (see Remark \ref{rem:ell-infty_c0} for instance). 
Nevertheless, we note for certain Banach spaces that a result similar to Proposition \ref{prop:almost} holds. 

\begin{prop}\label{prop:almost2}
Let $X$ be a Banach space such that $X \myeq X \oplus_\infty \mathbb{R}$. Then $X$ admits almost $\Delta$-points if and only if $\sup_{x \in S_X} \dec (x) = 2$. 
\end{prop}

\begin{proof}
It suffices to show that for $x \in B_X$ there exists $z \in S_X$ such that $\dec(z) \geq \dec(x)$. Let $x \in B_X \setminus \{0\}$ be given. We claim that $\dec((x,1)) \geq \dec(x)$. Let $\eps>0$, $\delta>0$ and $(x^*, \beta) \in S_{X^* \oplus_1 \mathbb{R}}$ with $(x,1) \in S:=  S\bigl( (x^*, \beta), \delta\bigr)$ be given. Put $z:= (-\|x\|^{-1} x, 1 )$ and observe that $z \in S$. Also, $\|z-(x,1) \| = 1+\|x\| \geq \dec(x) > \dec(x) -\eps$. Next, assume that $x^* \neq 0$. Note that we have
$$
\frac{x^*}{\|x^*\|} (x) > \frac{1}{\|x^*\|}(1-\delta -\beta).
$$
Thus, there exists $y \in B_X$ such that $\|x^*\|^{-1} x^* (y) > \|x^*\|^{-1} (1-\delta -\beta)$ and $\|y-x\| > \dec(x) - \eps$. Set $z:=(y,1)$ and observe that $z \in S$ since $(x^*, \beta) (y,1) > 1-\delta-\beta+\beta= 1-\delta$. It is clear that
$$
\|z- (x,1)\| = \|y-x\| > \dec(x) -\eps.
$$
This proves that $\dec( (x,1)) \geq \dec(x)$. 
\end{proof} 

Recall from \cite{MPR2023} that a point $x \in B_X$ is called a \emph{super Daugavet point} if $\sup_{y \in V} \|x-y\| = 1+\|x\|$ for every nonempty relatively weakly open subset $V$ of $B_X$. Thus, it is also natural to consider the \emph{super Daugavet constant} of a point $x$ in $B_X$, denoted by $\sdc(x)$. That is, $\sdc(x)$ is the constant given by
$$
\sdc(x) := \sup \Set{ \alpha \geq 0 | \begin{array}{l} \forall \eps>0 \text{ and a relatively weakly open } V \subseteq B_X, \\
\exists \, y \in V \text{ such that } \|y-x\| > \alpha - \eps \end{array}}.
$$
By definition, it is clear that $\dc(x) \geq \sdc(x)$ for every $x \in B_X$.
Although we introduced the notion of super Daugavet constant, we will focus on Daugavet and $\Delta$-constants in the following sections. 

In \cite{RZ2018}, the Daugavet index of thickness of a Banach space $X$, denoted by $\mathcal{T}(X)$, is introduced. Namely, 
\[
\mathcal{T} (X) = \inf \Set{ r>0 | \begin{array}{c} \exists \, x \in S_X \text{ and a relatively weakly open } W \text{ of } B_X \\ 
\text{ such that } \O \neq W \subseteq B(x,r):=\{ u \in B_X : \|u-x\|\leq r \} \end{array}}.
\]
Afterwards, a related Daugavet index $\mathcal{T}^s (X)$ is introduced in \cite{HLLNR} as follows:
\[
\mathcal{T}^s (X) = \inf \Set{ r>0 | \begin{array}{c} \exists \, x \in S_X \text{ and a slice } S \text{ of } B_X \\ 
\text{ such that } S \subseteq B(x,r)\end{array}}.
\]
The following shows that those indexes are indeed the infimum of (super) Daugavet constant of points over the unit sphere. 

\begin{prop}\label{prop:thick}
Let $X$ be a Banach space. Then
\begin{enumerate}
\itemsep0.3em
\item[\textup{(a)}] $\inf_{x\in S_X} \dc (x) = \mathcal{T}^s (X)$. 
\item[\textup{(b)}] $\inf_{x\in S_X} \sdc(x) = \mathcal{T}(X)$. 
\end{enumerate}
\end{prop}

\begin{proof}
As the statements (a) and (b) can be proved in a very similar way, we only prove (a). Suppose that $\dc(x)<r$ for some $x \in S_X$. Then there exist $\eps >0$ and $S=S(x^*,\delta)$ such that $\|x-y\| \leq r-\eps$ for every $y \in S$. This implies that $\mathcal{T}^s (X) \leq r$. Conversely, suppose that $\mathcal{T}^s (X)<r$. Let $\eps >0$ be such that $\mathcal{T}^s (X) < r- \eps <r$, and find $\mathcal{T}^s (X) < t < r-\eps$ such that there exist $x \in S_X$ and $S=S(x^*,\delta)$ satisfying $S \subseteq B(x,t)$. That is, $\|x-y\| \leq t < r-\eps$ for every $y \in S$. This implies that $\dc(x)<r$. 
\end{proof}  

\begin{remark}\label{remark:thick-delta}
It is also natural to think of the $\Delta$-version of the Daugavet index of thickness, which was once considered in \cite{HLLNR}. Given a Banach space $X$, let us define
\[
\mathcal{T}^\delta (X) :=  \inf \Set{ r>0 | \begin{array}{c} \exists \text{ a slice } S \text{ of } B_X \text{ and } x \in S \cap S_X \\ 
\text{ such that } S \subseteq B(x,r)\end{array}}. 
\]
Notice that $\mathcal{T}^s (X) \leq \mathcal{T}^\delta (X)$ and, arguing as in Proposition \ref{prop:thick}, observe the following: 
\[
\inf_{x\in S_X} \dec(x) = \mathcal{T}^\delta (X). 
\] 
\end{remark} 

One may ask whether it can happen that $\inf_{x \in S_X} \dc (x) = 0$ while $\dc (x) > 0$ for every $x \in S_X$, and a similar question for the $\Delta$-constant. The following shows that there is indeed such a Banach space $X$, and its construction relies on the fact \cite[Theorem 2.7]{HLLNR}.

\begin{theorem}
There exists a Banach space $X$ such that $\inf_{x \in S_X} \dec (x) = 0$ while $\dc (x) > 0$ for every $x \in S_X$. 
\end{theorem}

\begin{proof}
The result \cite[Proposition 2.8]{HLLNR} guarantees that for any $n \in \mathbb{N}$, there exists a Banach space $X_n$ satisfying that $\mathcal{T}^s (X_n) = 1/2^n$. As a matter of fact, a careful investigation of the proof shows that $\mathcal{T}^s (X_n)$ in fact coincides with $\mathcal{T}^\delta (X_n)$. Observe also that $\mathcal{T}^\delta (E_1 \oplus_1 E_2)\leq \mathcal{T}^\delta (E_1)$ for any Banach spaces $E_1$ and $E_2$, which is a consequence of Proposition \ref{prop:ell_1-converse} in Section \ref{sec:stability} and Remark \ref{remark:thick-delta}.

Now, set $X = \bigl(\oplus_n X_n\bigr)_{\ell_1}$. Then we have that 
\[
0\leq \mathcal{T}^s (X) \leq \mathcal{T}^\delta (X) \leq \min \{ \mathcal{T}^\delta (X_n) , \mathcal{T}^\delta (Y_n)\} \leq \frac{1}{2^n}, 
\]
where $Y_n = \bigl(\oplus_{k \neq n} Y_k\bigr)_{\ell_1}$ for each $n \in \N$. This shows that $\mathcal{T}^\delta (X) =0$. 

Next, we claim that $\dc(x)>0$ for every $x \in S_X$. Let $x= (x_n) \in S_X$ be given. If there exists $m \in \N$ such that $\|x_m\| = 1$, then we have 
\[
x = (0,\ldots, 0, x_{m}, 0, \ldots ) =: (x_m, 0) \in X_m \oplus_1 Y_m. 
\]
Note that $\dc(x)=\dc(x_m, 0) \geq \dc(x_m) \geq 2^{-m}$ (see Proposition \ref{prop:ell1sum} for its justification). 
So, suppose that $\|x_n\| < 1$ for every $n \in \N$. Let $y^*= (y_n^*) \in S_{X^*}$, $\eps>0$, $\delta >0$ be given. Find $k \in \N$ so that $\|y_k^* \| > 1- \delta$; so there is $y_k \in S_{X_{k}}$ such that $y_k^* (y_k)>1-\delta$. Letting $\iota_k$ be the embedding from $X_k$ into $X$, observe that $\iota_k (y_k) \in S(y^*, \delta)$ and
\[
\| \iota_k (y_k) - x \| = \|y_k - x_k\| + \sum_{n\neq k } \|x_n\| \geq 1 - \|x_k\| + (1- \|x_k\|) \geq 2 -2 \sup_{n \in \N} \|x_n\|. 
\]
This implies that $\dc (x) \geq 2 - 2 \sup_{n \in \N} \|x_n\| >0$. 
\end{proof} 

\section{Daugavet and $\Delta$-constants of points in some classical spaces}\label{sec:classical}

\subsection{On the Banach space $c_0$}

The following example shows that even a Banach space with the Radon-Nikod\'ym property (indeed, a finite dimensional Banach space) may have only finitely many elements whose Daugavet constant is zero.

\begin{prop}\label{prop:ell-infty-n}
Let $N \in \mathbb{N}$. For every $x = (x_n)_{n=1}^N \in B_{\ell_\infty^N}$, we have 
$$
\dc(x) = \max \{1 - |x_n| : n = 1,\ldots, N\}. 
$$
In particular, when $N=2$ we have 
\[
\dec(x) = \dc(x) = \max \{1 - |x_1|, 1-|x_2| \}. 
\]
\end{prop}

\begin{proof}
Let $x = (x_n)_{n=1}^N \in B_{\ell_\infty^N}$ be given. 
Let $\eps>0, \delta>0$ and $y^*=(y_n^*)_{n=1}^N \in S_{\ell_1^N}$ be given.
If we let $y:=(\sgn (y_n^*) )_{n=1}^N \in S_{\ell_\infty^N}$, then $y^*(y)=1$; hence $y \in S(y^*,\delta)$. 
Notice that 
\[
|x_n - \sgn (y_n^*)| \geq 1 - |x_n|.
\]
Thus, $\| x- y\| \geq \max \{ 1-|x_n| : n=1,\ldots, N\}$. This argument shows that $\dc (x) \geq \max \{ 1-|x_n| : n=1,\ldots, N\}$. 

Next, assume to the contrary that $\dc (x) > \max \{ 1-|x_n| : n=1,\ldots, N\}$. Without loss of generality, we may assume that $x_i \geq 0$ for every $i=1,\ldots, N$. 
Then we can find $\eta >0$ such that $x_i > 1 - \eta > 1 - \dc(x)$ for every $i =1,\ldots, N$. Let $\delta>0$ be sufficiently small so that $\delta < N^{-1} \eta$, and $y^* = N^{-1} (1,\ldots, 1) \in S_{\ell_1^N}$. If $y = (y_1,\ldots, y_N) \in S(y^*, \delta)$, then by a standard argument we have 
\[
y_i > 1 - N \delta > 1- \eta \quad \text{for } i =1,\ldots, N.
\]
It follows that $|x_i - y_i| <\eta$ for every $i=1,\ldots, N$, that is, $\|x-y\| < \eta$. This implies that $\dc(x) \leq \eta < \dc(x)$, which is a contradiction. 

Finally, when $N =2$, it suffices to show that $\dec(x) \leq \max \{ 1 -|x_1|, 1-|x_2|\}$. Without loss of generality, we may assume that $x_1 \geq x_2 \geq 0$. Thus, $\max \{ 1 -|x_1|, 1-|x_2|\} = 1-x_2$.
Again, assume to the contrary that $\dec(x) > 1-x_2$. Find $\eta >0$ so that $x_2 > 1 - \eta > 1 - \dec(x)$.  
If $y^* = (1/2, 1/2) \in S_{\ell_1^2}$, then 
\[
y^*(x) = \frac{x_1+x_2}{2} > \frac{x_1}{2} + \frac{1-\eta}{2} =: 1 - \delta.
\]
Observe that if $y=(y_1,y_2) \in S(y^*, \delta)$, then 
\[
y_1 + y_2 > x_1 + (1-\eta). 
\]
This implies that 
\[
x_1 -y_1 < \eta \, \text{ and } \, x_2 -y_2 \leq x_1 - y_2 < \eta. 
\] 
If $y_1 > x_1$, then $|x_1-y_1| = y_1 - x_1 \leq 1 -x_1 < \eta$. Also, if $y_1 \leq x_1$, then $|x_1 - y_1| = x_1 - y_1 < \eta$; hence $|x_1 - y_1|<\eta$ in any case. Similarly, $|x_2 - y_2| < \eta$, therefore $\|x-y\|< \eta$. Since $y \in S(y^*, \delta)$ is chosen arbitrarily, we obtain that $\|x-y\| < \eta$ which contradicts that $\dec(x) > \eta$.
\end{proof}

It should be noted that computing the precise value of the $\Delta$-constant of points in $\ell_{\infty}^N$ is significantly more difficult, even for ${N}=3$.
 In particular, it is not true in general for $N \geq 3$ that $\dec(x) = \max \{1-|x_n|: n = 1,...,N\}$ when $x \in B_{\ell_\infty^N}$ (see Example \ref{rem:ell-infty_c0} and Figure \ref{figure_dc}).

As an infinite dimensional version of Proposition \ref{prop:ell-infty-n}, we shall see that the Daugavet constant of a point $x$ in the space $c_0$ is precisely computed as follows, which can be considered as an improvement of \cite[Lemma 6.5]{ALMP}.

\begin{prop}\label{prop:dc-c_0}
For every $x \in B_{c_0}$, we have $\dc(x)=1$.
\end{prop}

\begin{proof}
We first show that $\dc(x) \geq 1$ for every $x \in B_{c_0}$. Observe that it suffices to prove that $\dc(x) \geq 1$ for every $x \in B_{c_{00}}$. Fix $x \in B_{c_{00}}$ and we may write $x = (x_1, \ldots, x_m, 0, \ldots)$ for some $m \in \N$. Let $\eps>0$, $\delta>0$ and $y^* = (y_n^*) \in S_{\ell_1}$ be given. Pick any $y = (y_n) \in S(y^*,\delta/3) \cap S_{c_0}$ and choose $k > m$ such that $|y_k^*| < \delta/3$. If we define $z = (z_n) \in S_{c_0}$ by $z_n = y_n$ for each $n \in \N \setminus \{k\}$ and $z_k = 1$, then
\begin{align*}
y^*(z) \geq \sum_{i \neq k} y_i^*y_i - \frac{\delta}{3} \geq y^*(y) - \frac{2\delta}{3} > 1- \delta,
\end{align*}
which yields $z \in S(y^*,\delta)$. The conclusion follows from $\|z-x\| \geq |z_k-x_k| =1$.

Now, suppose that $\dc(x_0) > 1$ for some $x_0 = (x_n) \in B_{c_0}$. Let $A \subseteq \N$ be the finite subset of all indices such that $|x_i| \geq (\dc(x_0) - 1)/4$ for each $i \in A$, where $A$ is nonempty since $\|x_0\| \geq \dc(x_0) - 1$. For 
\[
\eps_0 = \frac{\dc(x_0)-1}{2}, \,\delta_0 = \frac{\dc(x_0)+1}{2|A|}\, \text{ and }\, y_0^* = \frac{1}{|A|} \sum_{i \in A} \sgn (x_i) e_i^* \in S_{\ell_1},
\]
there must exist $y_0 = (y_n) \in S(y_0^*,\delta_0)$ such that
$$
\|y_0-x_0\| > \dc(x_0) - \eps_0 = \frac{1}{2} \left( \dc(x_0) + 1 \right). 
$$
However, observe that for each $i \notin A$, we have
$$
|y_i - x_i| \leq \frac{1}{4} \left(\dc(x_0) -1 \right) + 1 < \frac{1}{2} \left( \dc(x_0) + 1 \right).
$$
It follows that there must exist $k \in A$ such that 
\[
|y_k - x_k | = \| y_0 - x_0\| > \frac{1}{2} \left( \dc(x_0) + 1 \right) > 1.  
\]
This, in particular, implies that $\sgn (y_k) \neq \sgn (x_k)$. 
Moreover, 
\[
|y_k| > \frac{1}{2} (\dc( x_0) - 1). 
\]
Thus we have
\begin{align*}
y_0^*(y_0) = \frac{1}{|A|} \sum_{i \in A} \sgn (x_i) y_i &= \frac{1}{|A|} \sum_{i \in A \setminus \{k\}} \sgn (x_i) y_i - \frac{1}{|A|} |y_k|  \\
&\leq \frac{|A|-1}{|A|} - \left( \frac{\dc(x_0)-1}{2|A|}\right) = 1 - \left( \frac{\dc(x_0)+1}{2|A|}\right). 
\end{align*} 
This contradicts that $y_0 \in S(y_0^*, \delta_0)$. 
\end{proof}

%For the case of $\Delta$-constant on $c_0$, we have a different situation here.

Recall from Proposition \ref{prop:almost2} and \cite[Lemma 6.4]{ALMP} that $\sup_{x \in S_{c_0}} \dec(x) = 2$. This, compared with Proposition \ref{prop:dc-c_0}, shows that the Daugavet constant and the $\Delta$-constant of a point behave totally differently. We present an estimation of the $\Delta$-constant for points in $c_0$. 

\begin{theorem}\label{theorem:dec-c_0}
Let $x = (x_n) \in B_{c_0}$. Define $f_n: [0,1] \to \R$ for each $n \geq 3$ by
$$
f_n(t) = \begin{cases}
1+|t| &\text{ if }\, 0 \leq |t| \leq 1 - \dfrac{2}{n} \\
(n-1)(1-|t|) &\text{ if }\, 1 - \dfrac{2}{n} \leq |t| \leq 1.
\end{cases}
$$
Then, we have
$$
\dec (x) \geq \min \{ f_n(x_{i_1}), \ldots, f_n(x_{i_n}) \}
$$
for every $n \geq 3$ and for every set of indices $\{i_1, \ldots, i_n\} \subseteq \mathbb{N}$.
\end{theorem}

\begin{proof}
We may assume $x_n \geq 0$ for each $n \in \N$ for convenience. Let $n_0 \geq 3$ and $\{i_1, \ldots, i_{n_0}\} \subseteq \mathbb{N}$ be given. For simplicity, let us say for some $n_1 \leq n_0$ that
\begin{itemize}
\itemsep0.3em
\item $i_n = n$ for $n =1, \ldots, {n_0}$;
\item $0 < x_n \leq 1 -\frac{2}{n_0}$ for $n = 1,\ldots,n_1$; 
\item $1 -\frac{2}{n_0} < x_n \leq 1$ for $n = n_1+1, \ldots, n_0$. 
\end{itemize}
Define $w_1, \ldots, w_{n_0} \in c_0$ by \tiny
\begin{align*}
w_1 &= \left( \phantom{....} -1\phantom{....}, \frac{n_0 x_{2} + 1}{n_0 - 1}, \cdots, \frac{n_0 x_{n_1} + 1}{n_0 - 1} , \phantom{...............}1\phantom{...............} , \phantom{..} 1 \phantom{..}, \cdots, \phantom{w..........}1\phantom{w..........} , x_{n_0 + 1}, x_{n_0 + 2}, \cdots \right), \\
w_2 &= \left( \frac{n_0 x_1 + 1}{n_0-1}, \phantom{....} -1 \phantom{....}, \cdots, \frac{n_0 x_{n_1} + 1}{n_0 - 1} , \phantom{...............}1\phantom{...............} , \phantom{..} 1 \phantom{..}, \cdots, \phantom{w..........}1\phantom{w..........} , x_{n_0 + 1}, x_{n_0 + 2}, \cdots \right), \\
&\phantom{..............}\vdots \phantom{............} \vdots \phantom{w.........} \ddots \phantom{..........} \vdots \phantom{.........................} \vdots \phantom{...................} \vdots \phantom{....} \ddots \phantom{...............} \vdots \\
w_{n_1} &= \left( \frac{n_0 x_1 + 1}{n_0-1}, \frac{n_0 x_{2} + 1}{n_0 - 1},  \cdots, \phantom{.....} -1 \phantom{.....}, \phantom{...............}1\phantom{...............}, \phantom{..} 1 \phantom{..}, \cdots, \phantom{w..........}1\phantom{w..........},  x_{n_0 + 1}, x_{n_0 + 2}, \cdots \right), \\
w_{n_1+1} &= \left( \frac{n_0 x_1 + 1}{n_0-1}, \frac{n_0 x_{2} + 1}{n_0 - 1}, \cdots, \frac{n_0 x_{n_1} + 1}{n_0 - 1} , n_0x_{n_1+1} - (n_0-1) , \phantom{..}1\phantom{..} , \cdots, \phantom{w..........}1\phantom{w..........} , x_{n_0 + 1}, x_{n_0 + 2}, \cdots \right), \\
&\phantom{..............}\vdots \phantom{............} \vdots \phantom{w.........} \ddots \phantom{..........} \vdots \phantom{.........................} \vdots \phantom{...................} \vdots \phantom{....} \ddots \phantom{...............} \vdots \\
w_{n_0} &= \left( \frac{n_0 x_1 + 1}{n_0-1}, \frac{n_0 x_{2} + 1}{n_0 - 1}, \cdots, \frac{n_0 x_{n_1} + 1}{n_0 - 1} , \phantom{w............}1\phantom{w............} , \phantom{..}1\phantom{..}, \cdots, n_0x_{n_0} - (n_0-1) , x_{n_0 + 1}, x_{n_0 + 2}, \cdots \right).
\end{align*}\normalsize
Then, we have $\frac{1}{n_0} \sum_{n=1}^{n_0} w_n = x$ and $\|w_n-x\| \geq f_{n_0}(x_n)$ for $n = 1,\ldots,n_0$. Indeed, one can see that $\|w_n-x\| \geq 1 + x_n$ for $n=1,\ldots,n_1$ and $\|w_n-x\| \geq (n_0-1)(1-x_n)$ for $n = n_1+1,\ldots,n_0$. This shows that $x \in \operatorname{co} \Delta_{2-\min \{ f_{n_0}(x_1), \ldots, f_{n_0}(x_{n_0}) \}} (x)$, and thus $\dec(x) \geq \min \{ f_{n_0}(x_1), \ldots, f_{n_0}(x_{n_0}) \}$.
\end{proof}

As a consequence of theorem \ref{theorem:dec-c_0} (or a similar technique), we obtain the following: 
%first, we estimate the lower bound of $\Delta$-constant for $x \in B_{\ell_\infty^3}$, . That is, $\dec(x) = \max \{1-|x_n| : n = 1,2,3\}$ fails to hold. Next, we reprove that $\sup_{x \in S_{c_0}} \dec(x) = 2$ and compute the exact value of $\dec(x)$ when the point $x \in B_{c_0}$ is of the form $(t,\ldots,t,0,0,\ldots)$.

\begin{rem}\label{rem:ell-infty_c0} 
\, 
\begin{enumerate}
\itemsep0.3em
\item 
Let $x = (x_1,x_2,x_3) \in B_{\ell_\infty^3}$ be given. Define $f: [0,1] \to \R$ by
$$
f(t) = \begin{cases}
1+|t| &\text{ if }\, 0 \leq |t| \leq \dfrac{1}{3} \\
2(1-|t|) &\text{ if }\, \dfrac{1}{3} \leq |t| \leq 1.
\end{cases}
$$
Then, we have
$$
\dec (x) \geq \max\Big\{ \min \{ f(x_1),f(x_2),f(x_3) \} , 1-|x_1|, 1-|x_2|, 1-|x_3| \Big\}.
$$
Note that this in particular shows that a version of Proposition \ref{prop:ell-infty-n} for the $\Delta$-constant does not hold (see Figure \ref{figure_delta_2/3}). 
\begin{figure}
    \centering
    \begin{minipage}{0.5\textwidth}
        \centering
        \includegraphics[width=0.9\textwidth]{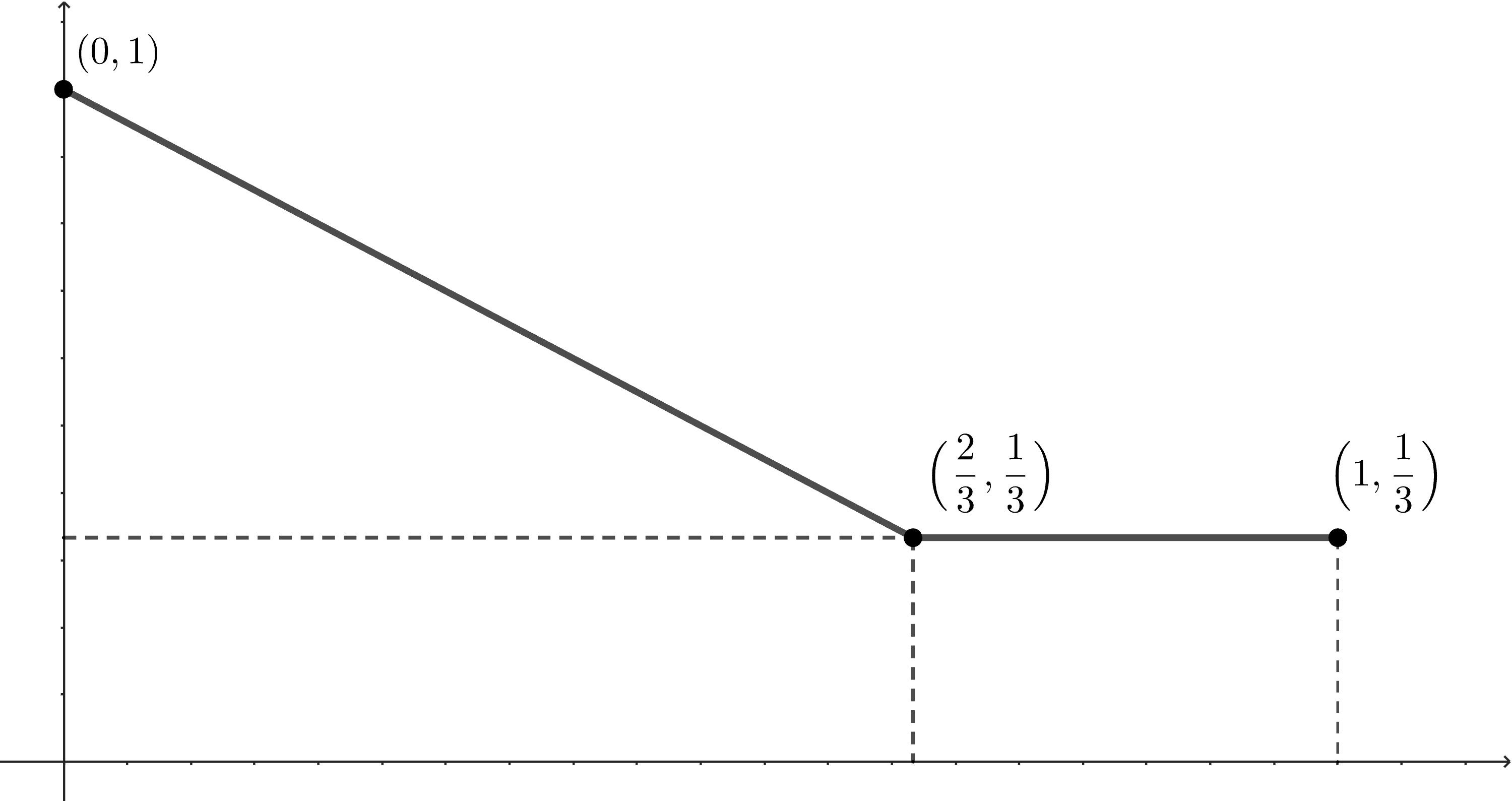} % first figure itself
        \caption{\\
        The value of \\
        $\dc \big(\big(\frac{2}{3}, \frac{2}{3}, t \big)\big)$, $0\leq t \leq 1$} \label{figure_dc}
    \end{minipage}\hfill
    \begin{minipage}{0.5\textwidth}
        \centering
        \includegraphics[width=0.9\textwidth]{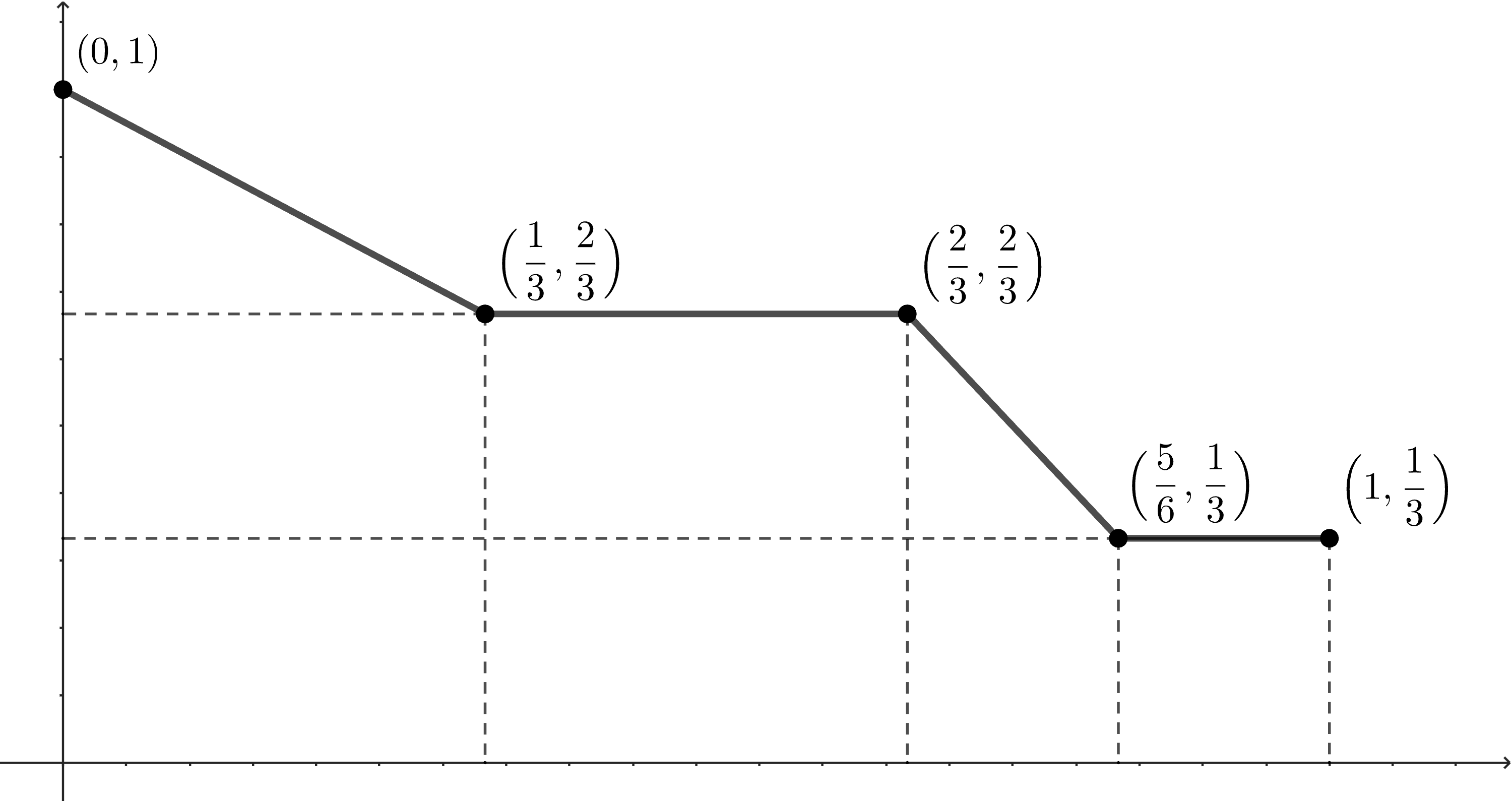} % second figure itself
        \caption{ \\
        The lower bound for \\
        $\dec \big(\big(\frac{2}{3}, \frac{2}{3}, t \big)\big)$, $0\leq t \leq 1$} \label{figure_delta_2/3}
    \end{minipage}
\end{figure}

\item Recall again from Proposition \ref{prop:almost2} that $\sup_{x \in S_{c_0}} \dec(x) = 2$. This can be shown explicitly by Theorem \ref{theorem:dec-c_0}: let $n \geq 3$ be given and consider 
\[
u = e_1 + \sum_{k=2}^{n+1} \left( 1- \frac{2}{n}\right) e_k \in S_{c_0}.
\]
A direct computation shows that the corresponding vector $z$ in Theorem \ref{theorem:dec-c_0} has norm $2-\frac{2}{n}$; thus, $\dec(u) \geq 2-\frac{2}{n}$. 
\item Let $n \in \N$ and $t \in [-1, 1]$ be given. Considering the element $\sum_{k=1}^n te_k \in B_{c_0}$, we have 
\begin{equation}\label{eq:ttt}
\dec \left(\sum_{k=1}^n te_k \right)= \min \Big \{ 1+|t| , \max \{ 1, (1-|t|)(n-1) \} \Big\}.  \tag{$\diamond$}
\end{equation}
Let us mention that this example shows that the mapping $x \in B_X \mapsto \dec (x)$ does not need to be Lipschitz (see Figure \ref{figure}). In order to obtain the equality in \eqref{eq:ttt}, we distinguish three cases. Without loss of generality, assume that $t \geq 0$. First, if $0\leq t\leq 1-\frac{2}{n}$, it is clear that $\dec \big(\sum_{k=1}^n te_k \big) \leq 1+t$. Second, suppose that $1-\frac{2}{n} \leq t \leq 1$. Consider 
\[
y^* = \Big(\frac{1}{n},\ldots, \frac{1}{n},0,\ldots \Big) \in S_{\ell_1}. 
\] 
Then $y^* \big(\sum_{k=1}^n te_k \big) = t > t -\delta$ for any $\delta>0$. By a standard convex combination argument, if $y=(y_n) \in S(y^*, 1-t+\delta)$, then $y_i \geq 1- (1-t+\delta)n$ for every $i=1,\ldots, n$. It follows that 
\begin{equation}\label{eq:ttt2}
\left\|y- \left (\sum_{k=1}^n te_k \right) \right\| \leq \max\{ 1, (1-t)(n-1) + n\delta\}.  \tag{$\dagger$}
\end{equation} 
Letting $\delta \rightarrow 0$, the right-hand side of \eqref{eq:ttt2} tends to $\max\{1,(1-t)(n-1)\}$. This implies that $\dec \big(\sum_{k=1}^n te_k \big)  \leq \max\{1,(1-t)(n-1)\}$.
\begin{figure}[h]
\includegraphics[scale=9]{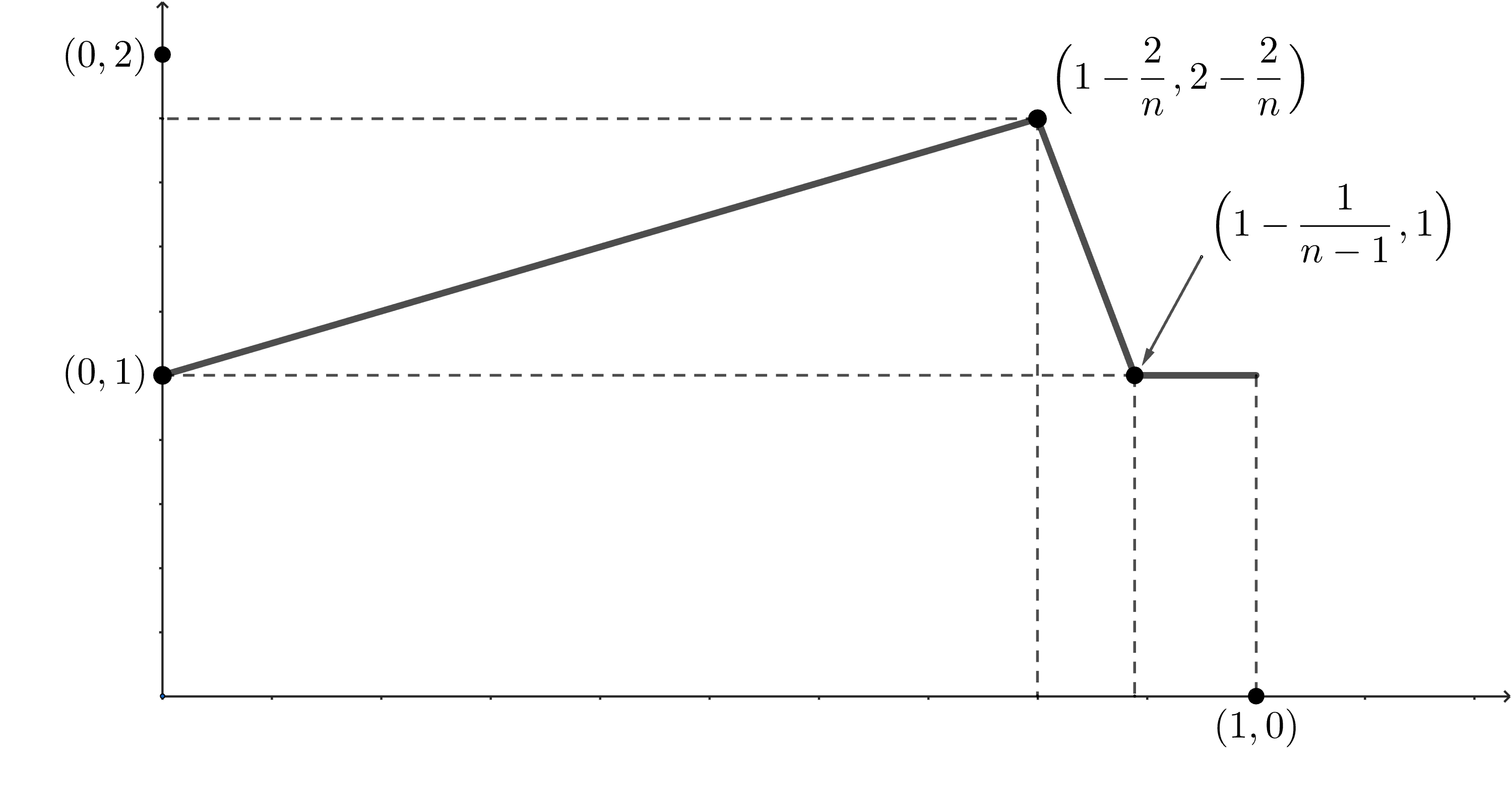}
\caption{The graph of $f(t)= \dec \big(\sum_{k=1}^n te_k \big)$, $0\leq t \leq 1$}     \label{figure}
\end{figure}
\end{enumerate} 
\end{rem}
%\vspace{0.5em}

We finish this section considering the space $\ell_\infty = \ell_1^*$. For a dual Banach space $X^*$, we can consider natural weak$^*$ versions of the Daugavet and $\Delta$-constant of a point by considering weak$^*$ slices instead of slices. For a point $x^* \in S_{X^*}$, the weak$^*$ Daugavet constant of $x$ is defined as: 
\[
\wdc(x^*) := \sup \Set{ \alpha \geq 0 | \forall \eps,\delta>0, y \in S_{X}, \exists \, y^* \in S(y,\delta) \text{ so that }  \|y^*-x^*\| > \alpha - \eps}
\]
where $S(y, \delta) = \{ z^* \in B_{Y^*} : z^*(y)>1-\delta\}$. The weak$^*$ $\Delta$-constant of $x$, denoted by $\wdec (x^*)$, can be defined in a similar way. It is clear that for a point $x^* \in B_{X^*}$, 
\begin{equation}\label{eq:weakstar2}
\dc(x^*) \leq \wdc(x^*) \leq \wdec(x^*) \text{ and } \dc(x^*) \leq \dec(x^*) \leq \wdec(x^*).
\end{equation} 
Moreover, if we write $J$ the canonical isometric embedding from $X$ into $X^{**}$, then 
\begin{equation}\label{eq:weakstar} 
%\dc( J(x)) \leq \dc(x) = \wdc(J(x)), 
\dc(x) = \wdc(J(x)). 
\end{equation} 
As a matter of fact, it is clear by definition that $\dc(x) \leq \wdc(J(x))$. For the reversed inequality, let $\eps >0$ and $S$ be a slice of $B_X$. Considering $S$ be a weak$^*$ slice of $B_{X^{**}}$, we find $z \in S_{X^{**}} \cap S$ such that $\|z-J_X (x)\| > \wdc(J_X (x))-\eps$. By Goldstine's theorem and using the lower weak$^*$ semicontinuity of the norm in dual spaces, we can find $u \in B_X \cap S$ such that $\|u-x\| > \wdc(J_X (x))-\eps$; hence $\dc(x) \geq \wdc(J_X (x))$. 

Combining the observations \eqref{eq:weakstar2} and \eqref{eq:weakstar}, we reprove the item (a) of Proposition \ref{prop:PLR} that $\dc(x) \geq \dc(J(x))$. We now present the following result concerning the weak$^*$ Daugavet constant of points in $\ell_\infty = \ell_1^*$. 

\begin{prop}
For every $x \in B_{\ell_\infty}$, we have $\wdc(x) = 1+ \limsup_n |x_n|$.
\end{prop} 

\begin{proof}
Let $x \in B_{\ell_\infty}$ be given. 
First, let $S=S(y,\delta)$ be a weak$^*$-slice of $B_{\ell_\infty}$ with $y=(y_n) \in S_{\ell_1}$. Find $z=(z_n) \in S_{c_0}$ so that $z \in S(y,\delta/3)$. Choose $m \in \N$ such that $\sum_{n=m+1}^\infty |y_n| < \delta/3$. Consider $w \in B_{\ell_\infty}$ given by 
\[
w = (z_1,\ldots, z_m, -\sgn(x_{m+1}), \sgn(x_{m+2}), \ldots ).
\] 
Then $\langle y, w\rangle \geq \sum_{n=1}^m y_n z_n - \sum_{n=m+1}^\infty |y_n| \geq \langle y, z \rangle - 2 \sum_{n=m+1}^\infty |y_n| > 1 - \delta/3 - 2\delta/3$; thus $w \in S(y,\delta)$. Moreover, 
\[
\|x-w\| \geq \sup_{n \geq m+1} | x_n + \sgn(x_n) | = 1+  \sup_{n \geq m+1} |x_n| \geq 1+ \limsup_n |x_n|. 
\]
This proves that $\wdc(x) \geq 1+\limsup_n |x_n|$. 

Next, assume that $\wdc(x) > 1+\limsup_n |x_n|$. Find $\delta >0$ sufficiently small so that 
\[
\wdc(x) - \delta > 1 + \limsup_n |x_n|. 
\]
Let $N_1 \in \N$ and consider the vector 
\[
y^{(1)} := \left(\frac{\sgn(x_1)}{N_1}, \ldots, \frac{\sgn(x_{N_1})}{N_1},0,0,\ldots\right) \in S_{\ell_1}. 
\]
Then there exists $z \in S(y^{(1)}, 1/N_1)$ such that $\|x-z\| > \wdc(x) - \delta > 1+\limsup_n |x_n|$. Note that if there exists $i \in \{1,\ldots, N_1\}$ so that $\sgn(x_i) \neq \sgn(z_i)$, then 
\[
\frac{N_1-1}{N_1} \geq \frac{1}{N_1} \sum_{n \in \{1,\ldots,N_1\}\setminus\{i\}} \sgn(x_n) z_n > \frac{1}{N_1} \sum_{n=1}^{N_1} \sgn(x_n) z_n > 1 - \frac{1}{N_1},
\]
which is a contradiction. Thus, $\sgn(x_i)=\sgn(z_i)$ for every $i=1,\ldots ,N_1$. In particular, this implies that there exists $j_1 > N_1$ such that $|x_{j_1} - z_{j_1}| > \wdc(x)-\delta$. Thus, $|x_{j_1}| > \wdc(x)-\delta-1$. Now, take $N_2 > j_1 > N_1$ and consider the element 
\[
y^{(2)} := \left(\frac{\sgn(x_1)}{N_2}, \ldots, \frac{\sgn(x_{N_2})}{N_2},0,0,\ldots\right) \in S_{\ell_1}. 
\]
Arguing in the same way, we end up with finding $j_2 > N_2$ such that $|x_{j_2}| > \wdc(x)-\delta-1$. In this way, we find a sequence of natural numbers $\{N_1 < j_1 < N_2 < j_2 < \cdots\}$. Passing to a subsequence, we may assume that $\lim_n |x_{j_n}|$ exists, so 
\[
1+ \lim_n |x_{j_n}| \geq \wdc(x) - \delta > 1 + \limsup_n |x_n|,  
\]
which is a contradiction. This shows that $\wdc(x) \leq 1 + \limsup_n |x_n|$ and completes the proof.
\end{proof} 

\subsection{On a Banach space ${L_1}$}

Let $\mu$ be a measure on some $\sigma$-algebra $\Sigma$ on a set $\Omega$. Recall that a set $A \in \Sigma$ is called an \textit{atom} if $0<\mu(A) < \infty$ and whenever $B \in \Sigma$ with $B \subseteq A$ satisfies $\mu(B) < \mu(A)$, then $\mu(B) = 0$. 
A $\sigma$-finite measure on $(\Omega, \Sigma)$ is called \textit{atomic} if every measurable set of positive measure contains an atom. In this case, there exists a countable partition of $\Omega$ formed by atoms up to a null set. In this subsection, we shall assume that all measures are $\sigma$-finite.

For simplicity, we denote by $L_1 (\mu)$ the space $L_1 (\Omega, \Sigma, \mu)$. 
Notice that a measurable function is almost everywhere constant on an atom. 
We start by giving the exact value of the Daugavet constant of an element in $L_1 (\mu)$ provided that $\mu$ is atomic. 

\begin{prop}\label{prop:dc-ell1}
Given a $\sigma$-finite atomic measure $\mu$ on $(\Omega, \Sigma)$, let $(A_n)$ be a countable partition of $\Omega$ formed by atoms up to a null set.  
For every $f \in {L_1 (\mu)}$ with $\|f\| \leq 1$, we have 
\[
\dc(f)= 1+\|f\| -2 \sup \left\{  |f \vert_{A_n} |\, \mu (A_n) : n \in \mathbb{N} \right\}.
\]
\end{prop}

\begin{proof}
Let $f \in L_1 (\mu)$ with $\|f\| \leq 1$ and $\delta >0$. 
Without loss of generality, let each $A_n$ be an atom for $\mu$. Note first that $|f\vert_{A_n}|| \mu(A_n)| \leq \|f\|$ for every $n \in \N$. 
Since $\mu$ is $\sigma$-finite, $L_1 (\mu)^* = L_\infty (\mu)$. Pick $g \in L_\infty(\mu)$ with $\|g \| = 1$. For each $n \in \N$, let $I_n : = \{ x \in A_n : |g(x)| > 1-\delta/2\}$. Then there exists $m \in \N$ such that $I_m$ has a positive measure, which implies that $\chi_{I_m} = \chi_{A_m}$ almost everywhere. 

Note that $h:= \frac{\sgn (g |_{A_m}) }{\mu(A_m)} \chi_{A_m}$ satisfies that 
\[
g(h) = \int_{A_m} \frac{\sgn (g |_{A_m}) }{\mu(A_m) } g \, d\mu =  \int_{I_m} \frac{ | g |_{A_m} | }{\mu(I_m) } \, d\mu \geq 1-\frac{\delta}{2} > 1- \delta.
\]
Observe that 
\begin{align*}
\|f-h\| &= \int_{A_m} \left | f- \frac{\sgn(g|_{A_m})}{\mu (A_m)} \right| \, d\mu + \int_{\cup_{n\neq m} A_n} |f| \, d\mu \\
&= | \sgn (g|A_m) - f|_{A_m} \mu (A_m)  | + \|f \| - |f|_{A_m}| \mu(A_m) \\ 
&\geq 1- |f|_{A_m}| \mu (A_m) + \|f \| - |f|_{A_m}| \mu(A_m) = 1+ \|f\| - 2 | f|_{A_m}| \mu(A_m).
\end{align*} 
This shows that $\dc(f) \geq 1+ \|f\| - 2 | f|_{A_m}| \mu(A_m)$.

Next, assume that 
\begin{equation}\label{eq:dc-ell1}
\dc(f) > 1+ \|f\| - 2 \sup \left\{  |f \vert_{A_n} |\, \mu (A_n) : n \in \mathbb{N} \right\}.
\end{equation}  
For simplicity, we may assume that $|f|_{A_1}| \mu(A_1) = \sup\left\{ |f|_{A_n}| \mu(A_n) : {n\in\N } \right\}$. If $|f|_{A_1}| \mu(A_1) = \|f\|$, then this would imply that 
\[
f =  {f \vert_{A_1}} \chi _{A_1}.
\] 
It is not difficult to check that $\dc (f) = 1 - \|f\|$ in such a case, which contradicts to \eqref{eq:dc-ell1}. 

Thus, we assume that $|f|_{A_1}| \mu(A_1) < \|f\|$. 
Take $S=S(\text{sgn}(f |_{A_1} ) \chi_{A_1} , \delta)$ with $\delta >0$ sufficiently small so that 
\[
\delta < \dc(f) - \bigl( 1+ \|f\| - 2 |f|_{A_1}| \mu(A_1)\bigr) \, \text{ and } \, \delta < \|f\| - |f|_{A_1}| \mu(A_1). 
\] 
Let us pick $h \in S$ with $\|h\|=1$ such that $\|f-h\| > \dc(f) - \delta$. That is, 
\[
\sgn ( f|_{A_1} ) h |_{A_1} \mu(A_1) > 1- \delta \geq \|f \| - \delta,
\]
and 
\[
\| f-h\| > \dc(f)-\delta > 1+ \|f\| - 2|f|_{A_1}| \mu(A_1). 
\]
Note in particular that $\sgn ( f|_{A_1} ) h |_{A_1} \mu(A_1) > |f|_{A_1}| \mu(A_1)$. 
Observe that 
\begin{align*}
\|f-h\| &= \int_{A_1} |f-h| \, d\mu + \int_{\cup_{n\geq 2} A_n} |f-h| \, d\mu \\
&= | \sgn ( f|_{A_1} ) h |_{A_1} \mu(A_1) - |f|_{A_1}|\mu(A_1)| + \int_{\cup_{n\geq 2} A_n} |f-h| \, d\mu \\ 
&\leq \sgn ( f|_{A_1} ) h |_{A_1} \mu(A_1) - |f|_{A_1}|\mu(A_1) + \|f\| - | f|_{A_1}| \mu(A_1)  + 1 - | h_{A_1}| \mu(A_1) \\
&\leq 1 + \|f\| - 2  | f|_{A_1}| \mu(A_1). 
\end{align*} 
which is a contradiction.
\end{proof}

Next, we obtain an upper bound for the $\Delta$-constant of an element $f \in L_1 (\mu)$ in terms of an atom $A$ contained in $\supp (f)$ and the idea of the proof comes from \cite[Theorem 3.1]{AHLP}, where it is proved that $\supp(f)$ does not contain any atom for $\mu$ if and only if $\dc(f) \text{ (or, } \dec(f)\text{)} = 2$.

\begin{prop}\label{prop:delta_L1}
Let $f \in L_1 (\mu)$ with $\|f\| \leq 1$. If $A$ is an atom in $\supp (f)$, then 
\[
\dec (f) \leq 1+\|f\|- 2|f \vert_{A}| \mu(A).
\] 
\end{prop} 

\begin{proof}
By considering $-f$ instead of $f$ when $c:= f \vert_A$ less than or equal to $0$, we may assume that $c \geq 0$. 
Assume to the contrary that $\dec (f) > 1+\|f\| -2c \mu(A)$. 
Let $\eps>0$ be such that $\eps < \dec(f) - ( 1+\|f\| - 2c \mu(A))$. Suppose that $g \in L_1 (\mu)$ satisfies that $\|g\| \leq 1$ and $\|f-g\| \geq \dec (f) - \eps$. Put $g \vert_A =: d$ and observe that 
\begin{align*} 
\dec(f) - \eps \leq \int_\Omega |f-g| \,d\mu &= \int_A |f-g|\,d\mu + \int_{\Omega\setminus A} |f-g|\,d\mu \\
&\leq \int_A |f-g|\,d\mu + \left( \|f\| -\int_A |f|  \, d\mu \right) + \left( 1 - \int_A |g| \, d\mu \right) \\
&\leq |c-d|\mu(A) + \|f\| + 1 - c \mu(A) - |d| \mu(A).
\end{align*} 
This implies that 
\begin{equation}\label{eq:mu(A)}
c\mu(A) + d\mu(A) \leq 1+\|f\| - \dec(f) +\eps + |c-d| \mu(A). 
\end{equation} 
if $c \leq d$, then \eqref{eq:mu(A)} implies that 
\[
2c \mu (A) \leq  1 + \|f\| -\dec (f) + \eps,
\]
which contradicts the choice of $\eps$. Thus, $c \geq d$, and we have from \eqref{eq:mu(A)} that 
\[
c \mu(A) + d\mu(A) \leq 1+\|f\| -\dec(f) + \eps + (c-d) \mu(A). 
\]
It follows that 
\begin{equation}\label{eq:d_ineq}
g\vert_A = d \leq \frac{1}{2\mu(A) } ( 1+ \|f\| -\dec(f) + \eps). 
\end{equation}

Now, by (c) of Proposition \ref{prop:dc-characterization}, we obtain that $f \in \overline{\co} \,\Delta_{2-\dec (f) + \eta}(f)$ for every $\eta >0$. In particular, $f \in \overline{\co} \,\Delta_{2-\dec (f) + \eps}(f) = \overline{\co} \,\{ g \in B_{L_1 (\mu)} : \|f-g\| \geq \dec (f) - \eps\}$.
However, if $g_1, \ldots, g_m \in \Delta_{2-\dec (f) + \eps} (f)$ and $\lambda_1,\ldots, \lambda_m$ are positive scalars so that $\sum_{i=1}^m \lambda_i = 1$, then 
\begin{align*}
\left\| f - \sum_{i=1}^m \lambda_i g_i \right \| &\geq \int_A \left | \sum_{i=1}^m \lambda_i (c- d_i) \right |\,d\mu \quad (d_i := g_i \vert_{A}, \, i=1,\ldots, m) \\ 
&= \int_A \sum_{i=1}^m \lambda_i (c- d_i) \,d\mu \quad (\text{since } c \geq d_i, \, i=1,\ldots, m) \\ 
&\geq \int_A c - \frac{1}{2\mu(A) } (  1+\|f\| -\dec(f) + \eps) \, d\mu \quad (\text{by } \eqref{eq:d_ineq}) \\
&= \frac{1}{2} \bigl[ 2c \mu(A) - (1+\|f\| -\dec(f) +\eps) \bigr] > 0.
\end{align*} 
This contradicts $f \in \overline{\co} \,\Delta_{2-\dec (f) + \eps}(f) $. 
\end{proof} 

\begin{remark}\label{rem:delta_L1}
Observe that there is another way to complete the proof of Proposition \ref{prop:delta_L1}. For instance, note that the assumption on $\eps >0$ is equivalent to say 
\[
\int_A f \, d\mu > \frac{1}{2} ( 1+\|f\| -\dec (f) + \eps). 
\]
Thus, by definition, there exists $g \in S_{L_1 (\mu)}$ such that  $\|f-g\| > \dec (f) - \eps$ and 
\[
d \mu (A) = \int_A g \, d\mu > \frac{1}{2} ( 1+\|f\| -\dec (f) + \eps), 
\] 
where $d = g \vert_A$, which contradicts \eqref{eq:d_ineq}. 
\end{remark} 

Combining Proposition \ref{prop:dc-ell1} and Proposition \ref{prop:delta_L1}, we obtain the following. 

\begin{cor}
Let $I = \mathbb{N}$ or $\{1,\ldots, N\}$ with $N \in \N$. 
For every $x = (x_n) \in B_{\ell_1 (I)}$, we have 
\[
\dc (x) = \dec (x)= 1+\|x\| -2\sup \{ |x_n| : n \in I \}.
\]
Moreover, 
\[
\wdc (x) = \wdec(x) = 1+\|x\| -2\sup \{ |x_n| : n \in I \}.
\]
\end{cor} 

\begin{proof}
The equalities for $\dc(x)$ and $\dec(x)$ are immediate from Proposition \ref{prop:dc-ell1} and Proposition \ref{prop:delta_L1}. For their weak$^*$ constants, it suffices to observe that $\wdec(x)\leq 1+\|x\| -2\sup \{ |x_n| : n \in I \}$. Indeed, note that in the case of $\ell_1$, the slice used in Remark \ref{rem:delta_L1} is nothing but a weak$^*$ slice determined by $e_k \in c_0$ for some $k \in \N$. 
\end{proof} 

\subsection{On a uniform algebra}

Given an infinite compact Hausdorff space $K$ and $f \in S_{C(K)}$, it is known that $f$ is a $\Delta$-point (or, Daugavet point) if and only if $|f(t)| = 1$ for some limit point $t \in K$ \cite[Theorem 3.4]{AHLP}. Afterwards, it is generalized in \cite[Theorem 3.2]{MR2022} to $L_1$-predual cases. 

Under certain assumptions about extreme points of $B_{X*}$, the following result provides an upper bound for the $\Delta$-constant of a point $x$ in a general Banach space $X$, allowing us to estimate an upper bound for the $\Delta$-constant of a certain element of uniform algebras.

%By following their argument step-by-step, then one can show that if there is no limit point $t \in K$ such that $|f(t)| = 1$, then 
%\[
%\dec(f) \leq 2 - \min \Big\{ 1-\max\{|f(t)|:t\in  H^c \}, \frac{2}{|H|} \Big\},
%\]
%where $H= \{t \in K : |f(t)| = 1\}$. However, arguing in a slightly sharper way, we can obtain the following. 

\begin{prop}\label{prop:general_extreme}
Let $X$ be a Banach space and $x \in S_X$. 
Suppose that 
\begin{equation}\label{eq:ext}
\exists \,\eps >0 \text{ such that } \{ e^* \in \ext (B_{X^*}) : |e^* (x)| > 1 - \eps\}  \text{ is finite. }
\end{equation} 
Then 
\begin{equation}\label{eq:l1predual}
\dec(x) \leq 1+ \sup \{|e^*(x)| : e ^* \in \ext (B_{X^*}), \, |e^*(x)| < 1 \} \,\,(< 2).  
\end{equation} 
\end{prop} 

\begin{proof}
By the assumption that the set in \eqref{eq:ext} is finite, and by the Krein-Milman theorem, observe that $\{ e^* \in \ext  (B_{X^*}) : e^* (x) = 1 \}$ is nonempty and finite. Let us say 
\[
\{ e^* \in \ext  (B_{X^*}) : e^* (x) = 1 \}  = \{e_1^*, \ldots, e_n^*\}.   
\]
It follows that 
\[
\alpha:= \sup \{|e^*(x)| : e^* \in \ext (B_{X^*}) \setminus \{\pm e_i^* :i=1,\ldots, n\} \} < 1.
\]
Consider the slice $S=S\bigl(g,\frac{1 - \alpha}{2n}\bigr)$ where $g = \frac{1}{n} \sum_{i=1}^n e_i^*$. If $y \in S$, then it is not difficult to check that $y \in S\bigl(e_i^*, \frac{1 - \alpha}{2}\bigr)$ for every $i =1,\ldots, n$. In particular, $|e_i^* (x -y)| < \frac{1 + \alpha}{2}$ for every $i =1,\ldots, n$. Thus, we observe that 
\begin{align*}
\|y-x\| &=\sup \{ |e^*(y-x)|: e^* \in \ext (B_{X^*})\} \\
&\leq \max \bigl\{ 1+\alpha, \max\{  |e_i^* (y-x)|: i=1,\ldots, n\} \bigr\}   = 1 +\alpha < 2. 
\end{align*} 
Since this holds for every $y \in S$, we conclude \eqref{eq:l1predual}.
\end{proof}

Let $K$ be an infinite compact Hausdorff space. We say $A \subseteq C(K)$ is a uniform algebra if it is a closed subalgebra of $A$ which separates points of $K$ and contains the constant functions. Recall that the \textit{Choquet boundary} for $A$ is defined by 
\begin{equation}\label{eq:partialA}
\partial A = \{ t \in A : \delta_t \vert_A \text{ is an extreme point of } B_{A^*} \}. \tag{$\blackdiamond$}
\end{equation} 

Given $f \in S_A$, notice from \eqref{eq:partialA} that the assumption \eqref{eq:ext} in Proposition \ref{prop:general_extreme} is equivalent to say that   
\begin{equation*}%\label{eq:C(K)_assumption} 
\exists \,\eps >0 \text{ such that } \{ t \in \partial A : |f(t)|> 1 - \eps \} \text{ is finite,} 
\end{equation*} 
which is, in turn, equivalent to that there exists no limit point $t \in \partial A$ so that $|f(t)| = 1$. In this case, $H := \{ t \in \partial A : |f(t)| = 1\}$ is a set of isolated points and 
\[
\sup \{ |f(t)| : t \in \partial A \setminus H \} = \max \{ |f(t)| : t \in \partial A \setminus H \} < 1. 
\]
Having this in mind, we obtain the following as a consequence of Proposition \ref{prop:general_extreme}. 
%In this case, the estimate \eqref{eq:l1predual} coincides with the one \eqref{eq:ck} in Proposition \ref{prop:ck}. 

\begin{cor}\label{cor:ck}
Let $A \subseteq C(K)$ be a uniform algebra and $f \in S_{A}$. Suppose that there is no limit point $t \in \partial A$ such that $|f(t)| = 1$. If $H=\{t \in K : |f(t)| = 1\}$, then 
\begin{equation*}\label{eq:ck}
\dec(f) \leq 1+ \max\{|f(t)|:t\in  \partial A \setminus H\} < 2.
\end{equation*}
\end{cor}

\subsection{On a Lipschitz-free space $\mathcal{F}(M)$}\label{subsec:lipfree}
Let us start by recalling some basic definitions about Lipschitz-free spaces. Given a pointed metric space $M$ with a distinguished point $0$, we denote by $\Lip (M)$ the space of all real-valued Lipschitz functions $f$ on $M$ which vanish at $0$ equipped with the norm 
\[
\|f\| = \sup \left \{ \frac{|f(x)-f(y)|}{d(x,y)}:x,y\in M, \, x\neq y \right\}. 
\] 
For each point $x \in M$, let $\delta_x$ be the evaluation functional on $\Lip (M)$, that is, $\delta_x (f) = f(x)$ for every $f \in \Lip (M)$. The norm closed linear span of $\{ \delta_x : x \in M \}$ in $\Lip (M)^*$ is called the \textit{Lipschitz-free space over $M$}, and denoted by $\mathcal{F}(M)$. We refer the reader to \cite{Weaver} for background on Lipschitz-free spaces. 

%It turned out in \cite[Theorem 2.1]{V2023} that the converse of \cite[Proposition 3.1]{JRZ} is true in Lipschitz-free spaces over general metric spaces. 
In the following, we investigate a quantitative version of the result \cite[Theorem 2.1]{V2023} by following the original idea and finally obtain the converse of Lemma \ref{lemma:dentdc} is valid in the case of Lipschitz-free spaces. Recall that $[u,v]_\delta := \{ p \in M : d(u, p) + d(v, p) < d(u,v)+\delta\}$.

\begin{theorem}\label{theorem:characterization_free}
Let $M$ be a metric space, $\mu \in S_{\mathcal{F} (M)}$ and $0 <\alpha \leq 2$ be given. Then the following are equivalent. 
\begin{enumerate} 
\itemsep0.3em
\item[\textup{(a)}] $\dc (\mu) \geq \alpha$,
\item[\textup{(b)}] for every $\nu \in \dent (B_{\mathcal{F} (M)})$, we have $\| \mu-\nu \| \geq \alpha$,
\item[\textup{(c)}] if for $u \neq v \in M$ and $r, s>0$ there exists $\delta >0$ such that 
\[
[u,v]_\delta \subseteq B(u,rd(u,v)) \cup B(v, sd(u,v))
\]
then $\| \mu - m_{u,v} \| \geq \alpha - 2 (r+s)$. 
\end{enumerate} 
\end{theorem}

Before we present the proof, we recall some necessary lemmas. 

\begin{lemma}\label{lemma_lip_free} 
Let $M$ be a metric space. 
\begin{enumerate}
\itemsep0.3em
\item[\textup{(a)}] \cite[Theorem 2.6]{JRZ} Let $\mu \in S_{\mathcal{F}(M)}$. Then for every $\eps >0$ there exists $\delta >0$ such that if $u\neq v \in M$ satisfy $d(u,v) < \delta$, then $\| \mu - m_{u,v} \| \geq 2-\eps$. 
\item[\textup{(b)}] \cite[Lemma 1.2]{V2023} Given $x,y,u,v \in M$ with $x\neq y$, $u\neq v$, and $\eps >0$, the following are equivalent: 
\begin{enumerate}
\itemsep0.3em
\item[\textup{(b1)}] $\| m_{x,y} + m_{u,v} \| \geq 2- \eps$,
\item[\textup{(b2)}] $d(x,v) + d(u,y) \geq d(x,y) + d(u,v) - \eps \max \{ d(x,y), d(u,v) \}$. 
\end{enumerate} 
Furthermore, the equalities in \textup{(a)} and \textup{(b)} hold simultaneously and in that case
\[
\| m_{x,y} + m_{u,v} \| = \frac{d(x,v)+d(u,y) + |d(x,y) - d(u,v)|}{\max \{ d(x,y) , d(u,v) \}}. 
\]
\item[\textup{(c)}] \cite[Lemma 2.4]{V2023} If $M$ is complete, $u,v \in M$ and $r,s,\delta>0$ with $r+s < d(u,v)$ are such that 
\[
[u,v]_\delta \subseteq B(u,r) \cup B(v,s), 
\]
then there exists $x \in B(u,r)$ and $y \in B(v,s)$ such that $m_{x,y}$ is a denting point. 
\end{enumerate} 
\end{lemma} 

\begin{proof}[Proof of Theorem \ref{theorem:characterization_free}]
(a) $\Rightarrow$ (b): It is a direct consequence of Lemma \ref{lemma:dentdc}. 

(b) $\Rightarrow$ (c): Arguing as in the proof of (ii) $\Rightarrow$ (iii) in \cite[Theorem 2.1]{V2023}, it is enough to prove the case when $M$ is complete. 
Let $u \neq v \in M$, $r,s >0$ such that there exists $\delta >0$ satisfying 
\[
[u,v]_\delta \subseteq B(u,rd(u,v)) \cup B(v, sd(u,v)). 
\] 
If $r+s \geq \alpha/2$, then there is nothing to prove; so we assume that $r+s < \alpha/2$. 
By Lemma \ref{lemma_lip_free}.(c), there exist $x \in B(u, r d(u,v))$ and $y \in B(v, s d(u,v))$ such that $m_{x,y}$ is a denting point. By the assumption, we then have that $\|\mu - m_{x,y} \| \geq \alpha$. Note that 
\[
d(u,x) + d(v,y) \leq r d(u,v) + s d(u,v) < \frac{\alpha}{2} d(u,v) \leq d(u,v). 
\]
Thus, we may find $\eta >0$ such that 
\[
d(x,u) + d(v,y) = d(u,v) + d(x,y) - \eta \max \{ d(x,y), d(u,v)\}. 
\] 
By applying Lemma \ref{lemma_lip_free}.(b), we obtain that 
\begin{align*}
 \|m_{x,y} - m_{u,v}\|&=\| m_{x,y} + m_{v,u} \| \\
 &= \frac{d(u,x)+d(v,y) + |d(u,v) - d(x,y)|}{\max \{d(u,v), d(x,y) \}} \\
 &\leq \frac{ 2d(u,x) + 2 d(v,y)}{\max \{ d(u,v), d(x,y) \}} \leq 2 \frac{rd(u,v)+s d(u,v)}{\max \{d(u,v),d(x,y)\}} \leq 2(r+s). 
\end{align*}
Consequently, 
\[
\| \mu - m_{u,v}\| \geq \| \mu - m_{x,y} \| - \| m_{x,y} - m_{u,v} \| \geq \alpha - 2(r+s). 
\] 

(c) $\Rightarrow$ (a): Fix $\eps >0$ and a slice $S= S(f,\xi)$ of $B_{\mathcal{F} (M)}$. Let $u_0 \neq v_0 \in M$ be such that $f(u_0)-f(v_0) > (1-\xi) d(u_0,v_0)$. By Lemma \ref{lemma_lip_free}.(a), there exists $\gamma >0$ such that if $x\neq y \in M$ satisfy $d(x,y) < \gamma$, then $\| \mu - m_{x,y} \| \geq 2-\eps$. 
Let $n \in \mathbb{N}$ and $\delta >0$ be such that 
\begin{itemize}
\itemsep0.3em
\item $f(u_0)-f(v_0) > (1-\xi) (1+\delta)^n d(u_0,v_0)$; 
\item $\left(1-\dfrac{\eps}{4} + \delta\right)^n d(u_0,v_0) < \gamma$. 
\end{itemize}
If $\| \mu - m_{u_0, v_0}\| \geq \alpha- \eps$, then there is nothing to prove. Suppose that $\| \mu - m_{u_0,v_0} \| < \alpha -\eps$. Then by the assumption, there exists 
\[
p \in [u_0,v_0]_{\delta d(u_0,v_0)} \setminus \left( B\left(u_0, \frac{\eps}{4}d(u_0,v_0)\right) \cup B\left(v_0,  \frac{\eps}{4} d(u_0,v_0)\right) \right). 
\]
This allows us to choose $u_1, v_1 \in M$ such that 
\begin{itemize}
\itemsep0.3em
\item $f(u_1)-f(v_1) > (1-\xi) (1+\delta)^{n-1} d(u_1,v_1)$; 
\item $d(u_1, v_1) < \left(1-\dfrac{\eps}{4}+\delta \right) d(u_0,v_0)$
\end{itemize}
(for its detail, see the proof of (iii) $\Rightarrow$ (i) in \cite[Theorem 2.1]{V2023}). Assume for $k \in \{ 1,\ldots, n-1\}$ that 
\begin{itemize}
\itemsep0.3em
\item $f(u_k)-f(v_k) > (1-\xi) (1+\delta)^{n-k} d(u_k,v_k)$; 
\item $d(u_k, v_k) < \left(1-\dfrac{\eps}{4}+\delta \right)^k d(u_0,v_0)$.
\end{itemize}
If $\| \mu - m_{u_k,v_k}\| \geq \alpha -\eps$ for some $k$, then we finish the proof. If not, we end up with the case:
\begin{itemize}
\itemsep0.3em
\item $f(u_n)-f(v_n) > (1-\xi)  d(u_n,v_n)$; 
\item $d(u_n, v_n) < \left(1-\dfrac{\eps}{4}+\delta \right)^n d(u_0,v_0) < \gamma$; implying that $\| \mu - m_{u_n, v_n}\| \geq 2- \eps$. 
\end{itemize}
Therefore, in any case, we find $u \neq v \in M$ so that $m_{u,v} \in S$ and $\| \mu - m_{u,v}\| \geq \alpha-\eps$. 
\end{proof} 

With the help of the structure of Lipschitz-free space, we can construct a Banach space in which a particular point exhibits diametrically opposed behaviors concerning its Daugavet constant and $\Delta$-constant. It should be noted that a similar discussion has already taken place in \cite[Section 3]{AAL+}, where the authors constructed a Banach space $X$ with a $\Delta$-point which is a limit of strongly exposed points of $B_X$ (thus, answering \cite[Question 7.13]{MPR2023} in the positive).

\begin{theorem}\label{theorem:deltanotdauga}
There exists a Banach space $X$ with a point $x \in S_X$ such that $\dec(x)=2$ and $\dc(x)=0$. 
\end{theorem}

\begin{proof}
Let us consider 
\[
M:= ( [0,1]\times \{0\}) \cup\left\{ \left(0,\frac{1}{n}\right),\left(1,\frac{1}{n}\right) : n \in \mathbb{N} \right\}\subseteq (\mathbb R^2,\Vert\cdot\Vert_2).
\]
Put $x:=(1,0)$, $y:=(0,0)$, $u_n = \bigl(1,\frac{1}{n}\bigr)$, and $v_n = \bigl(0,\frac{1}{n}\bigr)$ for each $n \in \N$. 
As there is a (1-Lipschitz) path connecting $x$ and $y$, by \cite[Proposition 4.2]{JRZ}, the molecule $m_{x,y}$ is a $\Delta$-point, i.e., $\dec (m_{x,y})=2$. 

Let us fix $n \in \N$, and note that $[u_n,v_n]=\{u_n,v_n\}$. Thus, it is clear to see that there exist sufficiently small $r,s,\delta >0$ with $r+s< d(u_n,v_n)=1$ such that 
\[
[u_n,v_n]_\delta := \{ p \in M : d(u_n, p) + d(v_n, p) < d(u_n,v_n)+\delta\} = \{u_n,v_n\},
\]
$B(u_n,r) = \{u_n\}$, and $B(v_n, s)=\{v_n\}$. Applying (3) of Lemma \ref{lemma_lip_free}, we obtain that $m_{u_n, v_n}$ is a denting point.

Suppose that $\dc (m_{x,y} ) >0$. Take $\alpha >0$ so that $\dc(m_{x,y})\geq \alpha >0$. By Theorem \ref{theorem:characterization_free}, it follows that we have $\| m_{u_n,v_n} - m_{x,y} \| \geq \alpha$ for every $n \in \N$. However, the distance between molecules can be estimated as follows (see \cite[Lemma 1.3]{ccgmr19}): 
\[
\| m_{u_n,v_n} - m_{x,y} \| \leq 2 \frac{d(x,u_n) + d(y, v_n) }{\max \{ d(x, y), d(u_n,v_n)\}} = \frac{4}{n} \rightarrow 0.
\]
This contradicts $\dc (m_{x,y}) \geq \alpha >0$; hence we conclude that $\dc(m_{x,y})=0$. 
\end{proof}

\section{Stability results}\label{sec:stability}

A norm $N$ on $\mathbb{R}^2$ is called \emph{absolute} if 
\[
N(a,b) = N(|a|,|b|) \quad \text{for all } (a,b)\in\mathbb{R}^2,
\]
and \emph{normalized} if $N(1,0)=N(0,1)=1$. 
Let $X$ and $Y$ be Banach spaces and $N$ an absolute normalized norm on $\mathbb{R}^2$. We denote by $X \oplus_N Y$ the product space $X \times Y$ endowed with the norm 
\[
\|(x,y)\|_N = N(\|x\|,\|y\|) \quad \text{for all } x\in X \text{ and } y \in Y. 
\]

The behavior of Daugavet and $\Delta$-points while taking direct sums with absolute normalized norm $N$ is first analyzed in \cite{AHLP}. Afterwards, the study of the stability for Daugavet and $\Delta$-points is continued in \cite{HPV}. In this section, we study how Daugavet and $\Delta$-constants behave under taking absolute sums. As the Daugavet and $\Delta$-constants have been shown to be closely related to the notions of certain Daugavet indexes of thickness (see Proposition \ref{prop:thick}), we find several results in \cite{HLLNR} concerning the stability of Daugavet indexes of thickness useful in our scenario. 
For simplicity, we shall consider only points in the sphere of $X\oplus_N Y$.

\subsection{Stability for the Daugavet constant}

We start with the following result, which is motivated by \cite[Proposition 2.2]{HLLNR}. The proof is similar, but we present the proof for the sake of completeness. 

\begin{prop}\label{prop:abs_N_ell_1}
Let $X$ and $Y$ be Banach spaces, $N$ be an absolute normalized norm on $\mathbb{R}^2$, and $\gamma>0$ is such that $N(\cdot) \geq \gamma \|\cdot\|_1$. Then we have the following: For $(x,y) \in S_{X \oplus_N Y}$, 
\[
\dc((x,y)) \geq 
\begin{cases}
2\gamma \left( \min \left\{ \dc \left( \dfrac{x}{\|x\|} \right) , \dc \left( \dfrac{y}{\|y\|}\right) \right\} - 1 \right) &\text{when } x \neq 0, y \neq 0, \\
2\gamma (\dc(x) - 1) &\text{when } y =0,\\
2\gamma( \dc(y)-1) &\text{when } x=0.
\end{cases} 
\]
\end{prop}

\begin{proof}
Assume first that $x\neq 0$ and $y\neq 0$. Assume also $\min \left\{ \dc \bigl( \frac{x}{\|x\|} \bigr) , \dc \bigl( \frac{y}{\|y\|}\bigr) \right\} > 1+\eps > 1$ for some $\eps >0$ (otherwise, there is nothing to prove).
Let $S=S\bigl((x^*,y^*),\delta\bigr)$ be a slice of $B_{X\oplus_N Y}$ with $\|(x^*,y^*)\| =1$. Take $(u,v) \in S_{X\oplus_N Y}$ so that $x^* (u) + y^* (v) > 1 - \delta' > 1-\delta$ for some $0<\delta'<\delta$. Fix $\eta >0$ so that $2\eta < \delta-\delta'$, and consider 
\[
S_1 = \begin{cases}
\bigl\{z \in B_X : x^*(z) > x^* \bigl(\frac{u}{\|u\|}\bigr) - \eta\bigr\} &\text{ when } u \neq 0 \\
B_X &\text{ when } u = 0 
\end{cases} 
\]
and 
\[
S_2 = \begin{cases}
\bigl\{w \in B_Y : y^*(w) > y^* \bigl(\frac{v}{\|v\|}\bigr) - \eta\bigr\} &\text{ when } v \neq 0 \\
B_Y &\text{ when } v = 0.
\end{cases} 
\]
Observe that there exist $z \in S_1$ and $w\in S_2$ such that $\bigl\|z - \frac{x}{\|x\|}\bigr\| > \dc\bigl( \frac{x}{\|x\|}\bigr) - \frac{\eps}{2}$ and $\bigl\|w - \frac{y}{\|y\|}\bigr\| > \dc\bigl( \frac{y}{\|y\|}\bigr) - \frac{\eps}{2}$. 
Note that 
\[
\left\|z - \frac{x}{\|x\|}\right\| > \dc\left( \frac{x}{\|x\|}\right) - \frac{\eps}{2} > (1+\eps) - \frac{\eps}{2} = 1 + \frac{\eps}{2}. 
\]
Then
\[
\bigl\|x - \|u\| z \bigr\| = \left\| \|x\| \frac{x}{\|x\|} + \|u\| (-z) \right\| \geq (\|x\| + \|u\|) \left (\dc\left( \frac{x}{\|x\|}\right) - \frac{\eps}{2} - 1\right) 
\]
where we used the fact: if $e_1, e_2 \in B_X$ satisfy $\|e_1+e_2\|\geq 1+\alpha$ for some $\alpha\in [0,1]$, then $\|\lambda e_1 + \mu e_2\| \geq (\lambda+\mu)\alpha$ for all $\lambda,\mu\geq 0$. Similarly, we obtain
\[
\bigl\|y-\|v\| w\bigr\|  \geq (\|y\| + \|v\|) \left (\dc\left( \frac{y}{\|y\|}\right) - \frac{\eps}{2} - 1\right). 
\]
Note that $(\|u\| z, \|v\| w) \in B_{X\oplus_N Y}$ and if $u \neq 0, v\neq 0$, then 
\begin{align*}
(x^*, y^*) (\|u\| z, \|v\| w) &= \|u\| x^*(z) + \|v\| y^* (w) \\
&>\|u\| \left( x^* \left(\frac{u}{\|u\|} \right)-\eta \right) + \|v\| \left( y^*\left(\frac{v}{\|v\|}\right) -\eta \right) \\
&> 1-\delta' - 2\eta > 1-\delta. 
\end{align*}
If one of $u$ and $v$ is $0$, say $v=0$, then $\|u\|=1$ and 
\begin{align*}
(x^*, y^*) (\|u\| z, \|v\| w) &= x^*(z)  > x^*(u) - \eta > 1-\delta' -\eta > 1-\delta. 
\end{align*}
Thus, in any case, $(\|u\| z, \|v\| w) \in S$. Now, observe that 
\begin{align*}
\| (x,y) &- (\|u\| z, \|v\| w) \|_N \\
&= N\bigl( \|x- \|u\|z \|, \|y - \|v\|w\|\bigr) \\
&\geq N\left( (\|x\| + \|u\|) \left (\dc\left( \frac{x}{\|x\|}\right) - \frac{\eps}{2} - 1\right), (\|y\| + \|v\|) \left (\dc\left( \frac{y}{\|y\|}\right) - \frac{\eps}{2} - 1\right) \right) \\ 
&\geq \min \left\{ \dc\left( \frac{x}{\|x\|}\right) - \frac{\eps}{2} - 1, \dc\left( \frac{y}{\|y\|}\right) - \frac{\eps}{2} - 1 \right\} N\bigl( \|x\|+\|u\|, \|y\|+\|v\|\bigr) \\
&\geq 2\gamma \min \left\{ \dc\left( \frac{x}{\|x\|}\right) - \frac{\eps}{2} - 1, \dc\left( \frac{y}{\|y\|}\right) - \frac{\eps}{2} - 1 \right\}.
\end{align*} 

Next, assume that $x=0$ or $y=0$. We only prove the case when $y=0$. Assume that $\dc(x) > 1+\eps$ for some $\eps >0$, otherwise there is nothing to prove. Let $S=S((x^*,y^*),\delta)$ be a slice of $B_{X\oplus_N Y}$ with $\|(x^*,y^*)\| =1$. Take $(u,v) \in S_{X\oplus_N Y}$ as above with some $0<\delta'<\delta$. Fix $\eta>0$ and consdier $S_1, S_2$ as above too. Find $z \in S_1$ such that $\|z - x \| > \dc(x) - \frac{\eps}{2}$. Let $w \in S_2 \cap S_Y$ be arbitrarily given. Arguing as above, $(\|u\| z, \|v\| w) \in S$. Moreover, 
\begin{align*}
\bigl\|(x,0) - (\|u\| z, \|v\| w) \bigr\|_N &= N\bigl(\|x - \|u\|z\| , \|v\|\bigr) \\
&\geq N \left( (1 + \|u\|)  \left(\dc (x) - \frac{\eps}{2} - 1\right), \|v\|\right) \\
&\geq   \left(\dc (x) - \frac{\eps}{2} - 1\right) N\bigl( 1+\|u\|, \|v\| \bigr) \\
&\geq 2\gamma \left(\dc (x) - \frac{\eps}{2} - 1\right).
\end{align*}  
This finishes the proof.
\end{proof}

When the absolute normalized norm $N$ on $\R^2$ in Proposition \ref{prop:abs_N_ell_1} is actually $\ell_1$-norm, then for $x \neq 0, y \neq 0$, we have 
\[
\dc((x,y)) \geq 2 \left( \min \left\{ \dc \left( \frac{x}{\|x\|} \right) , \dc \left( \frac{y}{\|y\|}\right) \right\} - 1 \right). 
\]
However, in the case of $\ell_1$-norm, the following improved lower bound for $\dc (x,y)$ can be obtained. 

\begin{prop}\label{prop:ell1sum}
Let $X$ and $Y$ be Banach spaces, and $(x,y) \in S_{X \oplus_1 Y}$. Then we have the following:
\begin{enumerate}
\itemsep0.3em
\item[\textup{(a)}] $\dc ((x,y)) \geq \min \left\{ \dc\left( \dfrac{x}{\|x\|}\right), \dc\left( \dfrac{y}{\|y\|}\right) \right\}$ when $x\neq 0$ and $y\neq 0$. 
\item[\textup{(b)}] $\dc ((x,y)) \geq \dc(y)$ when $x=0$.
\item[\textup{(c)}] $\dc ((x,y)) \geq \dc (x)$ when $y = 0$.  
\end{enumerate} 
\end{prop} 

\begin{proof} (a): Let $S=S\bigl( (x^*, y^*), \delta\bigr)$ be a slice of $B_{X\oplus_1 Y}$ with $\|(x^*,y^*)\|=1$ and $\eps >0$. Without loss of generality, we may assume that $\|x^*\| = 1$. 
Find $u \in S(x^*, \delta)$ so that $\bigl\|u- \frac{x}{\|x\|}\bigr\| \geq \dc \bigl(\frac{x}{\|x\|} \bigr) - \eps$. Note that $(u,0) \in S$ and 
\begin{align*}
\|(u,0)-(x,y)\|  = \|u-x \| + \|y\| &\geq \left\| u - \frac{x}{\|x\|}\right\| - \left\| \frac{x}{\|x\|} - x \,\right\| + \|y\| \\
&\geq \dc\left(\frac{x}{\|x\|}\right) - \eps - (1-\|x\|) + \|y\| \\
&= \dc\left(\frac{x}{\|x\|}\right) -\eps. 
\end{align*} 
This implies that $\dc ((x,y)) \geq \dc \bigl(\frac{x}{\|x\|}\bigr)$. Notice that if $\|y^*\| = 1$, then, by changing the role of $x$ and $y$, we observe $\dc ((x,y)) \geq \dc \bigl(\frac{y}{\|y\|}\bigr)$. 

As the proofs of (b) and (c) are very similar, we only prove (b): Let $S=S\bigl( (x^*, y^*), \delta\bigr)$ be a slice of $B_{X\oplus_1 Y}$ and $\eps >0$. Suppose that $\|x^*\| = 1$. Then there exists $u \in S_X$ so that $x^*(u)>1-\delta$. In other words, $(u,0) \in S$. In this case, $\|(0,y)-(u,0)\| = \|u\| +\|y\| = 2$. If $\|y^*\| = 1$, then we can pick $v \in S(y^*, \delta)$ so that $\| v-y\| > \dc (y)-\eps$. This implies that $(0,v) \in S$ and $\|(0,y)-(0,v)\| = \|v-y\| > \dc(y)-\eps$. In any case, we obtain a point in $S$ which is a distance of $\dc(y)-\eps$ from $(0,y)$. This proves that $\dc ((0,y))\geq \dc (y)$. 
\end{proof}

If the norm $N$ on $\R^2$ in Proposition \ref{prop:abs_N_ell_1} is $\ell_\infty$-norm, then we may obtain the result on $\dc((x,y))$ for $(x,y)\in S_{X\oplus_\infty Y}$ by taking $\gamma = \frac{1}{2}$. However, the following direct computation produces a better lower bound.

\begin{prop}
Let $X$ and $Y$ be Banach spaces, and $(x,y) \in S_{X \oplus_\infty Y}$. Then we have the following:
\[
\dc ((x,y)) \geq 
\begin{cases}
\dc(x) &\text{when } \|x\|=1, \\
\dc(y) &\text{when } \|y\| =1. 
\end{cases} 
\]
\end{prop}

\begin{proof}
We only prove the case when $\|x\|=1$: Let $S=S( (x^*, y^*), \delta)$ be a slice of $B_{X\oplus_\infty, Y}$ with $\|(x^*, y^*)\| =1$ and $\eps >0$. Note that $\|x^*\| +\|y^*\| =1$. If $x^* \neq 0$, then find $u \in S\bigl(\frac{x^*}{\|x^*\|}, \delta\bigr)$ such that $\|u-x\| > \dc(x)-\eps$. Pick $v \in S\bigl(\frac{y^*}{\|y^*\|}, \delta\bigr)$ if $y^* \neq 0$ and $v \in B_Y$ arbitrarily if $y^* =0$. In any case, $(x^*,y^*) (u,v) = x^*(u) +y^*(v) > 1-\delta$, that is, $(u,v) \in S$. Note that 
\[
\|(x,y)-(u,v)\| = \max \bigl\{\|x-u\| ,\|y-v\|\bigr\} > \dc(x)-\eps. 
\]
It remains to consider the case when $x^* = 0$. In this case, $\|y^*\| =1$. Then $(u,v) \in S$ for any $u \in B_X$ and $v \in S(y^*, \delta)$. Note that 
\[
\|(x,y)-(-x,v)\| = \max \bigl\{\|2x\| , \|y-v\| \bigr\} =2. 
\]
Consequently, whether $x^*$ is $0$ or not, we can always find a point in $S$ which is ($\dc(x)-\eps$)-away from $(x,y)$. So, $\dc((x,y)) \geq \dc(x)$. 
\end{proof} 

In \cite{HPV}, the authors generalized the notion of \textit{positively octahedrality} \cite{HLN} of an absolute norm on $X$ by introducing the so-called \textit{$A$-octahedrality} for a subset $A$ of $S_X$. With this notion, they were able to complete describing the class of absolute norms that admit a Daugavet point. 
Following their ideas, we will provide a lower bound for the Daugavet constant of a point lying in an absolute sum $X\oplus_N Y$ given that the absolute norm $N$ is $A$-OH.

\begin{definition}
Let $X$ be a Banach space and $A \subseteq S_X$. We say the norm on $X$ is \textit{$A$-octahedral} (for short, $A$-OH) if for every $x_1,\ldots, x_n \in A$ and $\eps >0$, there exists $y \in S_X$ such that $\|x_i + y \| \geq 2-\eps$ for every $i=1,\ldots, n$. 
\end{definition}

Let $N$ be an absolute normalized norm on $\R^2$. Put $c = \max \{s : N(s,1)=1\}$ and $d = \max \{ t: N(1,t) = 1 \}$, and define $A = \{(c,1), (1,d)\}$. We will consider from now on this specific set $A$. If an absolute norm $N$ is $A$-\textup{OH}, then, by definition, we can find $(a,b) \in \mathbb{R}^2$ with $N(a,b) =1$, $a,b \geq 0$ such that 
\begin{equation}\label{eq:OHab}
N\bigl((c,1)+(a,b)\bigr)=2 \quad \text{and} \quad N\bigl((1,d)+(a,b)\bigr) = 2. \tag{$\star$}
\end{equation}

\begin{prop}\label{prop:A-OH}
Let $X$ and $Y$ be Banach spaces, $x \in S_X$, $y \in S_Y$, and let $N$ be an $A$-\textup{OH} norm with $(a,b)$ as in \eqref{eq:OHab}. Then $(ax,by) \in S_{X \oplus_N Y}$ satisfies that 
\[
\dc ((ax,by)) \geq 2 \bigl( \min \bigl\{\dc(x), \dc(y)\bigr\} -1\bigr). 
\]
\end{prop} 

\begin{proof}
Let $(x^*, y^*) \in S_{(X\oplus_N Y)^*}, \alpha >0$ and $\eps >0$. Choose $\delta >0$ so that $\delta N(1,1) < \eps$. Find $u \in B_X$ and $v \in B_Y$ such that 
\[
x^* (u) \geq \left( 1- \frac{\alpha}{2}\right) \|x^*\| \quad \text{and} \quad y^*(v) \geq \left( 1-\frac{\alpha}{2}\right) \|y^*\|, 
\]
$\|x-u\| > \dc(x)-\delta$ and $\|y-v\| > \dc(y)-\delta$. Find $k \geq 0, \ell \geq 0$ such that $N(k,\ell)=1$, $N\bigl((a,b)+(k,\ell)\bigr)=2$, and $k\|x^*\| + \ell \|y^*\| = 1$. Note that 
\[
(x^*,y^*)(ku,\ell v) = kx^*(u) +\ell y^*(v) \geq \left(1-\frac{\alpha}{2}\right) (k\|x^*\| +\ell \|y^*\| ) = 1-\frac{\alpha}{2},
\]
which implies that $(ku, \ell v) \in S( (x^*,y^*), \alpha)$. By triangle inequality, we obtain $\|ax-ku\| \geq (a+k)(\dc(x)-\delta-1)$ and $\|by-\ell v\| \geq (b+\ell) (\dc(y) -\delta-1)$. Thus
\begin{align*}
\| (ax,by) - (ku,\ell v)\|_N &= N\bigl(\|ax-ku\|, \|by- \ell v\| \bigr) \\
&\geq N\bigl( (a+k)(\dc (x)-\delta-1), (b+\ell)(\dc(y)-\delta-1) \bigr) \\
&\geq N\bigl( (a+k)(\dc (x)-1), (b+\ell)(\dc(y)-1) \bigr) - 2 \eps \\
&\geq \min \bigl\{ \dc(x)-1, \dc(y)-1 \bigr\} N(a+k, b+\ell) - 2 \eps \\
&=2 \min \bigl\{ \dc(x)-1, \dc(y)-1 \bigr\}  - 2 \eps,
\end{align*}
as desired.
\end{proof} 

Arguing very similarly, we can obtain the following result in the special case when $a=0$ or $b=0$. 

\begin{prop}\label{prop:A-OH2}
Let $X$ and $Y$ be Banach spaces, $x \in S_X$, $y \in S_Y$, and let $N$ be an $A$-\textup{OH} norm with $(a,b)$ as in \eqref{eq:OHab}. 
\begin{enumerate}
\itemsep0.3em
\item[\textup{(a)}] If $a=0$, then $\dc((0,y)) \geq 2(\dc(y) -1)$.
\item[\textup{(b)}] If $b=0$, then $\dc((x,0))  \geq 2 (\dc (x)-1)$. 
\end{enumerate} 
\end{prop}

\begin{rem} Note that $\ell_1$- and $\ell_\infty$-norm are $A$-\textup{OH}. 
For instance, if we consider $\ell_1$-norm on $\R^2$, then the corresponding set $A$ turns out to be $\{ (0,1), (1,0)\}$. In this case, any point $(a,b) \in \R^2$ with $a+b=1, a,b \geq 0$ satisfies \eqref{eq:OHab}. Given $(x,y) \in S_{X\oplus_1 Y}$, writing $(x,y) = \bigl(\|x\|\frac{x}{\|x\|}, \|y\|\frac{y}{\|y\|}\bigr)$ if $x\neq0, y\neq 0$, Proposition \ref{prop:A-OH} and \ref{prop:A-OH2} show that 
\[
\dc((x,y)) \geq \begin{cases}
2 \left( \min \left\{ \dc \left( \dfrac{x}{\|x\|} \right) , \dc \left( \dfrac{y}{\|y\|}\right) \right\} - 1 \right) &\text{when } x \neq 0, y \neq 0. \\
2 (\dc(x) - 1) &\text{when } y =0.\\
2( \dc(y)-1) &\text{when } x=0.
\end{cases} 
\] 
%However, observe that it is not as sharp as Proposition \ref{prop:ell1sum}. 

In the case of $\ell_\infty$-norm on $\R^2$, given $(x,y) \in S_{X\oplus_\infty Y}$, one can obtain a similar lower bound for $\dc ((x,y))$. 
\end{rem}

We finish this subsection by presenting the following result which gives an upper bound for $\dc(x,y)$. 
It can be obtained by simply arguing as in \cite[Proposition 2.4]{HLLNR}, so we omit the proof. 

\begin{prop}\label{prop:upperbound}
Let $X$ and $Y$ be Banach spaces, $N$ be an absolute normalized norm on $\mathbb{R}^2$, and $\Gamma>0$ is such that $N(\cdot)\leq \Gamma \|\cdot\|_\infty$. Then we have the following: 
\begin{enumerate}
\itemsep0.3em
\item[\textup{(a)}] If $(0,1)$ is an extreme point of $B_{(\R^2, N)}$, then $\dc( (x,0)) \leq \Gamma$ for every $x \in S_X$.
\item[\textup{(b)}] If $(1,0)$ is an extreme point of $B_{(\R^2, N)}$, then $\dc( (0,y)) \leq \Gamma$ for every $y \in S_Y$. 
\end{enumerate} 
In particular, for $1<p<\infty$, $\dc((x,0)) \leq 2^{1/p}$ and $\dc ((0,y))\leq 2^{1/p}$ whenever $(x,0) \in S_{X \oplus_p Y}$ and $(0,y) \in S_{X \oplus_p Y}$.
\end{prop} 

As a consequence of Proposition \ref{prop:upperbound} combined with Proposition \ref{prop:abs_N_ell_1}, we have that for $1<p<\infty$, $x \in S_X$ and $y \in S_Y$
\[
2^{1/p} (\dc (x)-1) \leq \dc ((x,0)) \leq 2^{1/p} \,\text{ and }\, 2^{1/p} (\dc (y)-1) \leq \dc ((0,y)) \leq 2^{1/p}. 
\]

\subsection{Stability for the $\Delta$-constant}

In the proof of Proposition \ref{prop:abs_N_ell_1}, if we start with a slice $S$ containing the given point $(x,y)$, then we may take $(u,v) = (x,y)$ and proceed as before. Thus, a similar result with respect to $\Delta$-constants can be obtained. However, we will see that even stronger claim holds in this case.

%\begin{prop}\label{prop:abs_N_ell_1_delta}
%Let $X$ and $Y$ be Banach spaces, $N$ an absolute normalized norm on $\mathbb{R}^2$, and $\gamma>0$ is such that $N(\cdot) \geq \gamma \|\cdot\|_1$. Then we have the following: For $(x,y) \in S_{X \oplus_N Y}$, 
%\[
%\dec ((x,y)) \geq 
%\begin{cases}
%2\gamma \left( \min \left\{ \dec \left( \dfrac{x}{\|x\|} \right) , \dec \left( \dfrac{y}{\|y\|}\right) \right\} - 1 \right) &\text{when } x \neq 0, y \neq 0; \\
%2\gamma (\dec(x) - 1) &\text{when } y =0; \\
%2\gamma( \dec(y)-1) &\text{when } x=0.
%\end{cases} 
%\]
%\end{prop}

Recall that the diametral local diameter two property (for short, DLD2P) behaves well under absolute sums, that is, for any absolute normalized norm $N$ on $\R^2$, both $X$ and $Y$ have the DLD2P if and only if $X \oplus_N$ has the DLD2P \cite[Theorem 3.2]{IK}. 
In this regard, it is natural to expect some stability result of the $\Delta$-constant. 

\begin{prop}\label{prop:absolute-dec}
Let $X$ and $Y$ be Banach spaces, $N$ an absolute normalized norm on $\mathbb{R}^2$, and $a,b\geq 0$ with $N(a,b) = 1$. Given $x \in B_X$ and $y \in B_Y$, we have 
\[
\dec ((ax,by)) \geq \min \bigl\{ \dec(x), \dec(y)\bigr\}. 
\]
Moreover, we have the following:
\begin{enumerate}
\itemsep0.3em
\item[\textup{(a)}] If $b=0$, then $\dec((x,0))  \geq \dec(x)$.
\item[\textup{(b)}] If $a=0$, then $\dec((0,y)) \geq \dec(y)$. 
\end{enumerate} 
\end{prop} 

\begin{proof}
Let $\eps >0$ and $0<\eta<\eps$. By Proposition \ref{prop:dc-characterization}, we can find $x_1, \ldots, x_m \in \Delta_{2-\dec(x) +\eps} (x)$ and $y_1,\ldots, y_m \in \Delta_{2-\dec(y)+\eps} (y)$ such that 
\[
\left\| x - \frac{1}{m} \sum_{i=1}^m x_i \right\| < \eta \, \text{ and } \, \left\| y - \frac{1}{m} \sum_{i=1}^m y_i \right\| < \eta,
\]
where we applied \cite[Lemma 4.1]{AHLP} to choose the same number of vectors in $X$ and $Y$. 
Note that 
\[
\left\| (ax,by) - \frac{1}{m} \sum_{i=1}^m (ax_i, by_i) \right\|_N = N \left( a \left\|x- \frac{1}{m} \sum_{i=1}^m x_i\right\| , b \left \| y- \frac{1}{m} \sum_{i=1}^m y_i\right\| \right) \leq \eta. 
\]
Moreover, for each $i =1,\ldots, m$, 
\[
\bigl\| (ax,by)-(ax_i,by_i) \bigr\|_N = N\bigl(a \|x-x_i\|, b\|y-y_i\| \bigr) \geq \min \bigl\{ \dec(x)-\eps, \dec(y)-\eps \bigr\}. 
\]
Thus, by again Proposition \ref{prop:dc-characterization}, we conclude that $\dec((ax,by)) \geq \min \bigl\{ \dec(x), \dec(y)\bigr\}$.

When $a=0$ or $b=0$, the above argument proves the estimates (a) and (b). 
\end{proof}

In case of $\ell_1$-sum or $\ell_\infty$-sum, we have the following direct consequences.

\begin{cor}\label{cor:ell_1-sum}
Let $X$ and $Y$ be Banach spaces. Then we have the following: 
\begin{enumerate}
\itemsep0.3em
\item[\textup{(a)}] For $(x,y) \in S_{X\oplus_1 Y}$, 
\[
\dec ((x,y)) \geq 
\begin{cases}
 \min \left\{ \dec \left( \dfrac{x}{\|x\|} \right) , \dec \left( \dfrac{y}{\|y\|}\right) \right\} &\text{when } x \neq 0, y \neq 0, \\
\dec(x)  &\text{when } y =0, \\
 \dec(y) &\text{when } x=0.
\end{cases} 
\]
\item[\textup{(b)}] For $(x,y) \in S_{X\oplus_\infty Y}$, 
\[
\dec ((x,y)) \geq \min \bigl\{ \dec(x), \dec(y)\bigr\}. 
\]
\end{enumerate} 
\end{cor}

\begin{proof}
For the assertion (a), we apply Proposition \ref{prop:absolute-dec} with $a = \|x\|$ and $b = \|y\|$. For (b), use again Proposition \ref{prop:absolute-dec} with $a = 1$ and $b=1$.
\end{proof}

If one of the components is zero in case of $\ell_1$-sum, then the above inequality is indeed an equality.

\begin{prop}\label{prop:ell_1-converse}
Let $X$ and $Y$ be Banach spaces. Then, $\dec((x,0)) = \dec(x)$ for every $(x,0) \in S_{X \oplus_1 Y}$.
\end{prop}

\begin{proof}
Thanks to Corollary \ref{cor:ell_1-sum}, it suffices to show for every $(x,0) \in S_{X \oplus_1 Y}$ that $\dec((x,0)) \leq \dec(x)$. Assume that $\dec(x) \leq \alpha$, then there exists a slice $S(x^*,\delta)$ with $x \in S(x^*,\delta)$ such that $\|y-x\| < \alpha$ for every $y \in S(x^*,\delta)$. Given $0<\eta<\delta$, by \cite[Lemma 2.1]{IK} we can find  $y_1^* \in S_{X^*}$ satisfying that
\[
x \in S(y_1^*,\eta) \subseteq S(x^*,\delta).
\]
Note that for $(u,v) \in S( (y_1^*, 0), \eta)$, we have $y_1^* (u) >  1-\eta$ and $\|v\| < \eta$. This implies that
\[
\|(u,v) - (x,0)\| = \|u-x\| + \|v\| < \|u-x\| + \eta,
\]
where $u, x \in S(x^*,\delta)$. As it is clear that $(x, 0) \in S( (y_1^*, 0) , \eta)$ and $\eta$ can be chosen arbitrarily small, we conclude that $\dec((x,0)) \leq \alpha$.
\end{proof}

\vspace{0.3em}

\noindent \textbf{Acknowledgment:} 
The authors are grateful to the anonymous referee for the careful reading of the manuscript and for providing a number of suggestions which have improved its final form. Mingu Jung was supported by a KIAS Individual Grant (MG086601) and by June E Huh Center for Mathematical Challenges (HP086601) at Korea Institute for Advanced Study.
\vspace{0.3em}

\noindent \textbf{Conflicts of interest:} There is no conflict of interest.
\vspace{0.3em}

\noindent \textbf{Data availability:} No data was used for the research described in the article.


\begin{thebibliography}{99}

\bibitem{AAL+} \textsc{T.A. Abrahamsen, R. Aliaga, V. Lima, A. Martiny, Y. Perreau, A. Prochazka and T. Veeorg}, \textit{Delta-points and their implications for the geometry of Banach spaces}, J. London Math. Soc., 109: e12913. https://doi.org/10.1112/jlms.12913

\bibitem{AHNTT} \textsc{T.A. Abrahamsen, P. Hájek, O. Nygaard, J. Talponen and S. Troyanski}, \textit{Diameter 2 properties and convexity}, Studia Math. \textbf{232} (2016), 227--242.

\bibitem{AHLP} \textsc{T.A.~Abrahamsen, R.~Haller, V.~Lima and K.~Pirk}, \textit{Delta- and Daugavet points in Banach spaces}, Proc. Edinb. Math. Soc. \textbf{63} (2020), 475--496.

\bibitem{ALMP} \textsc{T.A.~Abrahamsen, V.~Lima, A.~Martiny and Y.~Perreau}, \textit{Asymptotic geometry and Delta-points}, Banach J. Math. Anal. \textbf{16}, 57 (2022).

\bibitem{ALMT} \textsc{T.A. Abrahamsen, V. Lima, A. Martiny and S. Troyanski}, \textit{Daugavet- and delta-points in Banach spaces with unconditional bases}, Trans. Amer. Math. Soc. Ser. B \textbf{8} (2021), 379--398. MR 4249632

	
	\bibitem{BLR18} \textsc{J.~Becerra~Guerrero, G.~L\'opez-P\'erez and A.~Rueda~Zoca}, \textit{Diametral diameter two properties in Banach spaces}, J. Convex Anal. \textbf{25} (2018), 817--840.


\bibitem{ccgmr19} \textsc{B. Cascales, R. Chiclana, L. Garc\'ia-Lirola, M. Mart\'in and A. Rueda Zoca}, \textit{On strongly norm attaining Lipschitz maps}, J. Funct. Anal. \textbf{277} (2019), 1677--1717. 


\bibitem{DJRZ} \textsc{S. Dantas, M. Jung and A. Rueda Zoca}, \textit{Daugavet points in projective tensor products}, Q. J. Math. 73 (2022), pp. 443--459.

\bibitem{D63} \textsc{I.K. Daugavet}, \textit{On a property of completely continuous operators in the space}, C. Uspekhi Mat. Nauk \textbf{18.5} (1963), 157--158 (Russian).

\bibitem{gppr2018}
	\textsc{L. Garc{\'{i}}a-Lirola, C. Petitjean, A. Proch{\'{a}}zka and A. Rueda Zoca},
	\emph{Extremal structure and duality of Lipschitz free spaces},
	Mediterr. J. Math. \textbf{15} (2018), 69.

\bibitem{HLN} \textsc{R. Haller, J. Langemets and R. Nadel}, \textit{Stability of average roughness, octahedrality, and strong diameter $2$ properties of Banach spaces with respect to absolute sums}, Banach J. Math. Anal. \textbf{12}(1) (2018), 222--239.

\bibitem{HLLNR} \textsc{R. Haller, J. Langemets, V. Lima, R. Nadel and A. Rueda Zoca} \textit{On Daugavet indices of thickness}, J. Funct. Anal. \textbf{280} (2021), 108846.

\bibitem{HLPV} \textsc{R. Haller, J. Langemets, Y. Perreau and T. Veeorg}, \textit{Unconditional bases and Daugavet renormings}, J. Funct. Anal. \textbf{286} (2024), 110421.

\bibitem{HPV} \textsc{R.~Haller, K.~Pirk and T.~Veeorg}, \textit{Daugavet- and Delta-points in absolute sums of Banach spaces}, J. Convex Anal. \textbf{28} (2021), 41--54.



\bibitem{IK} \textsc{Y. Ivakhno and V. Kadets}, \textit{Unconditional sums of spaces with bad projections}, Visn. Khark. Univ., Ser. Mat. Prykl. Mat. Mekh. \textbf{645} (2004), 30--35.

\bibitem {ikw} \textsc{Y.~Ivakhno, V.~Kadets and D.~Werner}, \textit{The Daugavet property for spaces of Lipschitz functions}, Math. Scand. \textbf{101}, 2 (2007), 261--279.

\bibitem{K96} \textsc{V.M. Kadets}, \textit{Some remarks concerning the Daugavet equation}, Quaestiones Math. \textbf{19} (1996), 225--235.

\bibitem{KW04} \textsc{V.~Kadets and D.~Werner}, \textit{A Banach space with the Schur and the Daugavet property}. Proc. Amer. Math. Soc. \textbf{132} (2004), 1765--1773.


\bibitem{KSSW} \textsc{V.M. Kadets, R. V. Shvidkoy, G. G. Sirotkin and D. Werner}, \textit{Banach spaces with the Daugavet property}, Trans. Amer. Math. Soc. \textbf{352} (2000), 855--873.

%\bibitem{KKW} \textsc{V. Kadets, N. Kalton, and D. Werner} \textit{Remarks on rich subspaces of Banach spaces}. Stud. Math. 159, 195--206 (2003). 

\bibitem{J23} \textsc{M. Jung}, \textit{Daugavet property of Banach algebras of holomorphic functions and norm-attaining holomorphic functions}, Adv. Math., \textbf{421} (2023), 109005.

\bibitem{JRZ} \textsc{M. Jung and A. Rueda Zoca}, \textit{Daugavet points and $\Delta$-points in Lipschitz-free spaces}, Stud. Math., \textbf{265} (2022), 37--55.

%\bibitem{LP} \textsc{J. Langemets and K. Pirk}, \textit{Stability of diametral diameter two properties}, RACSAM \textbf{115}, 96 (2021).

\bibitem{LLT86} \textsc{B.L. Lin, P.K. Lin and S.L. Troyanski}, \textit{Some geoemtric and topological properties of the unit sphere in a Banach space}, Math. Ann. \textbf{274} (1986), 613--616. 

\bibitem{Loz} \textsc{G. Ya. Lozanovskii}, \textit{On almost integral operators in KB-spaces}, Vestnik Leningrad Univ. Mat. Mekh. Astr. \textbf{21.7} (1966), 35--44 (Russian).


\bibitem{MR2022} \textsc{M.~Mart\'in and A. Rueda Zoca}, \textit{Daugavet property in projective symmetric tensor products of Banach spaces}, Banach J. Math. Anal. \textbf{16}, 35 (2022).

\bibitem{MPR2023}
\textsc{M. Mart\'in, Y. Perreau and A. Rueda Zoca}, \textit{Diametral notions for elements of the unit ball of a Banach space}, Dissertationes Math (2024). DOI:10.4064/dm230728-21-3

\bibitem{RZ2018} \textsc{A. Rueda Zoca}, \textit{Daugavet property and separability in Banach spaces}, Banach J. Math. Anal. \textbf{12} (2018) 68--84.

\bibitem{V2023} \textsc{T.~Veeorg}, \textit{Characaterizations of Daugavet points and delta-points in Lipschitz-free spaces}, Studia Math. \textbf{268} (2023), 213--233.

\bibitem{V_lipschitzfunction} \textsc{T. Veeorg}, \textit{Daugavet- and delta-points in spaces of Lipschitz functions}, arXiv:2206.03475


\bibitem{Weaver} \textsc{N. Weaver}, \textit{Lipschitz algebras} (Second edition), World Scientific Publishing Co., Inc., River Edge, NJ, 2018.

\bibitem{W96} \textsc{D. Werner}, \textit{An elementary approach to the Daugavet equation}, Interaction between functional analysis, harmonic analysis, and probability (Columbia, MO, 1994), 449--454, Lecture Notes in Pure and Appl. Math. \textbf{175}, Dekker, New York, 1996.

\bibitem{W97} \textsc{D. Werner}, \textit{The Daugavet equation for operators on function spaces}, J. Funct. Anal. \textbf{143} (1997), 117--128.

\bibitem{W01} \textsc{D. Werner}, \textit{Recent progress on the Daugavet property}, Irish Math. Soc. Bull. \textbf{46} (2001), 77--97.

\bibitem{Wo92} \textsc{P. Wojtaszczyk}, \textit{Some remarks on the Daugavet equation}, Proc. Amer. Math. Soc. \textbf{115} (1992), 1047--1052.

\end{thebibliography}
\end{document}